\documentclass[a4paper, 11pt, reqno]{amsart}
\usepackage[left=3cm,right=3cm,top=3.5cm,bottom=3.5cm]{geometry}
\usepackage[foot]{amsaddr}

\usepackage{mlmodern}
\usepackage[T1]{fontenc}

\usepackage{enumitem}
\usepackage{amsmath}
\usepackage{dsfont}
\usepackage{mathrsfs}
\usepackage[vcentermath]{youngtab}
\usepackage{tensor}

\usepackage[dvipsnames]{xcolor}
\usepackage[colorlinks=true,
			linkcolor=BlueViolet,
			citecolor=ForestGreen,
			urlcolor=Cerulean]{hyperref}
			
\usepackage{thmtools}
\newtheorem{thm}{Theorem}
\newtheorem{cor}{Corollary}
\newtheorem{defn}{Definition}
\newtheorem{lem}{Lemma}
\newtheorem{prop}{Proposition}
\theoremstyle{remark}
\declaretheorem[name=Remark,qed={\lower-0.3ex\hbox{$/$\!\!$/$}}]{rmk}
\declaretheorem[name=Example,qed={\lower-0.3ex\hbox{$/$\!\!$/$}}]{ex}

\usepackage{MnSymbol}

\newcommand\longto{\ensuremath{\longrightarrow}}
\renewcommand\to{\ensuremath{\longto}}

\newcommand\R{\ensuremath{\mathbb{R}}}
\newcommand\SL{\ensuremath{\mathrm{SL}}}
\newcommand\Aff{\ensuremath{\mathrm{Aff}}}
\newcommand\la{\ensuremath{\langle}}
\newcommand\ra{\ensuremath{\rangle}}

\newcommand\scrL{\ensuremath{\mathscr{L}}}
\newcommand\tr{\ensuremath{\mathrm{tr}\,}}
\newcommand\grad{\ensuremath{\mathrm{grad}}}
\newcommand\X{\ensuremath{\mathfrak{X}}}
\newcommand\Riem{\ensuremath{\mathrm{Riem}}}
\newcommand\Ric{\ensuremath{\mathrm{Ric}}}
\newcommand\Scal{\ensuremath{\mathrm{Scal}}}
\renewcommand\P{\ensuremath{\mathsf{P}}} 
\newcommand\J{\ensuremath{\mathsf{J}}} 
\newcommand\Id{\ensuremath{\mathds{1}}}
\renewcommand\div{\ensuremath{\mathrm{div}}}

\usepackage{stmaryrd}

\newcommand\Sym{\ensuremath{\mathrm{Sym}}}

\usepackage{upgreek}
\usepackage{relsize}
\newcommand{\bigtau}{\ensuremath{\tau}}

\usepackage{euscript,yfonts}
\newcommand{\sha}{\EuScript{X}}
\newcommand{\shta}{\mathsf X}

\begin{document}

\title{Affine hypersurfaces and superintegrable systems}
\author{Vicente Cort\'es, Andreas Vollmer}
\address{\scriptsize University of Hamburg, Department of Mathematics, Bundesstraße~55, 20146 Hamburg, Germany\newline
\indent{\normalfont andreas.vollmer@uni-hamburg.de, vicente.cortes@uni-hamburg.de} }

\begin{abstract}
	It was recently shown that under mild assumptions second-order conformally superintegrable systems can be encoded in a $(0,3)$-tensor, called \emph{structure tensor}. For \emph{abundant systems}, this approach led to algebraic integrability conditions that essentially allow one to restore a system from the knowledge of its structure tensor in a point on the manifold.
	Here we study the geometric structure formalising such systems, which we call an \emph{abundant} manifold. The underlying Riemannian manifold is necessarily conformally flat.

	We establish a correspondence between these superintegrable systems and the geometry of affine hypersurfaces.
	More precisely, we show that abundant manifolds correspond to certain non-degenerate relative affine hypersurfaces normalisations in $\R^{n+1}$ 
	($n\ge 2$).
	We also formulate the necessary and sufficient conditions non-degenerate relative affine hypersurface normalisations in $\R^{n+1}$ need to satisfy, if they arise from abundant manifolds. These relative affine hypersurface normalisations are called abundant hypersurface normalisations.
	Both for abundant manifolds and for relative affine hypersurface normalisations a natural concept of conformal equivalence can be defined. We prove that they are compatible, permitting us to identify conformal classes of abundant manifolds with abundant hypersurface immersions (without specified normalisation).
\end{abstract}

\maketitle


\section{Introduction}

The purpose of this paper is to relate two a priori unrelated structures.
On the one hand, there are affine hypersurface normalisations, a basic concept in affine differential geometry.
On the other, there are abundant manifolds, which formalise a certain class of second-order (maximally) superintegrable systems.
Before entering into precise definitions in Section~\ref{sec:preliminaries}, this current section provides some background and a first glance at the results.

An affine hypersurface is essentially a subset of an affine space, say $\mathbb R^{n+1}$, described by an immersion $f:M\to\mathbb R^{n+1}$ of an $n$-dimensional manifold $M$. Two such hypersurfaces are (affinely) congruent, if they can be mapped one onto the other by an affine transformation. An affine hypersurface normalisation is an affine hypersurface together with a (nowhere vanishing) transversal field $\xi:M\to\mathbb R^{n+1}$.
Affine hypersurface normalisations are a classical subject of study in geometry that was pioneered by Wilhelm Blaschke \cite{BlaschkeII}. It has applications in various areas of mathematics, such as in information geometry, mathematical physics and differential geometry.
A prominent example are affine spheres, which have significance for Calabi-Yau manifolds and for Monge-Amp\`ere equations, e.g.~\cite{Calabi, nomizu-sasaki,ACG2007,Martinez2005,CY1986}, as well as in special Kähler geometry, e.g.~\cite{BC2001, BC2003}.
Affine spheres also appear in relation to soliton theory and the \c{T}i\c{t}eica equation \cite{DP09,IU2018}, and its discretisation \cite{BS1999,Schief2000}.
Affine hypersurfaces are a key element in the theory of statistical manifolds, see e.g.\ \cite{information-geometry,Shima,Matsuzoe2010,Opozda2019,Opozda2021}.
Classification results exist for affine hypersurfaces normalisations with special properties, such as a parallel Pick tensor, see e.g.\ \cite{DV1991,DVY1994,LW1997,HLLV2011,HLV2011}.

Superintegrable systems likewise constitute a classical subject of investigation in mathematical physics. Loosely speaking, they are Hamiltonian systems with more integrals of motion than needed for integrability in Liouville's sense. The latter property allows one to solve Hamilton's equations of motion by quadrature.
We require the integrals of motion to be second-order polynomials in the momenta coordinates and we assume that there is a maximal number of linearly independent integrals of motion. These systems are called \emph{abundant} (second-order) superintegrable systems and give rise to a specific geometric structure, which we call an \emph{abundant manifold}. It encodes all the information of the original mechanical system \cite{KSV2023,KSV2024,KSV2024_bauhaus}. In particular, these works have established algebraic integrability conditions that essentially allow one to restore the abundant system from the knowledge of its structure tensor in a point on $M$.

Superintegrability is not preserved under conformal rescalings, which has been remedied by the introduction of the more general \emph{conformally} superintegrable systems. In these systems, the integrals of motion are replaced by functions that are constant along Hamiltonian trajectories on the zero locus of the Hamiltonian.

Abundant superintegrable systems have been classified in dimensions two and three, cf.\ \cite{Evans1990,KKPM2001,KKM-1,KKM-2,KKM-3,KKM-4,Kress07}. Algebraic geometry has been employed in this context \cite{MPW13,Capel&Kress,Capel_phdthesis,Kress&Schoebel}.
Separability has been an important property in the development of the theory of second-order superintegrable systems, cf.\ \cite{KKM2018} for instance.
The associated algebras of Killing tensors for abundant superintegrable systems have also received considerable attention, e.g.\ \cite{Capel&Kress&Post,KM2014,Kress07}, leading to a relation with hypergeometric orthogonal polynomials organised in the Askey-Wilson scheme \cite{KMP07,KMP11,KMP13}.
Advances have recently been made concerning conformally superintegrable systems in dimensions higher than three \cite{KSV2023,KSV2024,Vollmer2024}. However, many questions on (abundant) superintegrable systems in higher dimensions remain open.
\medskip

This paper contributes to the understanding of abundant systems by establishing a bridge, for arbitrary dimension, that allows one to transfer methods from the theory of affine differential geometry to superintegrable systems. It may in the future be used to construct new examples of superintegrable systems, or to classify them, for instance.
\smallskip

We give a first flavor of the results obtained in this paper. Our first key result is the  correspondence between abundant manifolds and a special class of affine hypersurface normalisations.
On one side of this correspondence we establish that (abundant second-order) conformally (maximally) superintegrable systems\footnote{%
	For brevity we suppress the attributes `abundant', `second-order' and `maximally' from now on and always require these properties.}
can be realised as affine hypersurface normalisations.
Let $(M,g)$ be a Riemannian oriented (connected) manifold of dimension $n\geq 2$.
Our starting point is the following observation for an abundant system on a Riemannian manifold $(M,g)$: it admits a natural $(0,3)$-tensor field $\mathcal T\in\Gamma(\mathrm{Sym}^2_0(T^*M)\otimes T^*M)$ 
that decomposes according to
\begin{equation}\label{eq:TSt}
	\mathcal T(X,Y,Z) = S(X,Y,Z)+T(X)g(Y,Z)+T(Y)g(X,Z)-\frac2n\,g(X,Y)T(Z),
\end{equation}
where $S$ is a totally symmetric and trace-free $(0,3)$-tensor and where $T$ is exact, see \cite{KSV2023,KSV2024}.
The tensor field $\mathcal T$, and thus the tensor field $S$ and the 1-form $T$, depend only on the space of compatible conformal Killing tensor fields of the system.
Since $T$ is exact, there locally exists a (smooth) function $t$ such that $T=dt$.
As it turns out, $S$ is conformally equivariant (to be explained later) and the function $t$ can be interpreted as a conformal scale function, see~\cite{KSV2024}.
In~\cite{KSV2023} a natural affine connection associated to $\mathcal T$ is introduced,
which becomes flat precisely for \emph{abundant} proper\footnote{%
	By the attribute \emph{proper}, we mean actual superintegrable systems that are not just conformally superintegrable.}
superintegrable systems.
This connection is defined by
\[
	\nabla^*_XY = \nabla^g_XY - \hat C(X,Y),
\]
where $\hat C$ denotes the $(1,2)$-tensor which corresponds to the 
$(0,3)$-tensor $C$,
\begin{equation}\label{eq:C2ST.intro}
	C(X,Y,Z) = \frac13\Bigl( S(X,Y,Z)+T(X)g(Y,Z)+T(Y)g(X,Z)+T(Z)g(X,Y) \Bigr)
\end{equation}
for $X,Y,Z\in\mathfrak X(M)$.
We find that it arises as the dual connection $\nabla^*$ of an immersed non-degenerate relative affine hypersurface normalisation.
This is, essentially, the pivotal insight of our correspondence. It yields a realisation of abundant systems as a (special type of) affine hypersurface.
As a consequence, we are able to reinterpret aspects of superintegrable systems geometrically.
For example, we find that the \emph{standard scale} gauge, which was introduced and discussed in \cite{KSV2024}, is the case of Blaschke immersions.
\smallskip

Let us now turn to the other side of our correspondence. We formulate the necessary and sufficient conditions that need to hold, if a relative affine hypersurface normalisation arises from an abundant manifold. In this case, we call it an abundant hypersurface normalisation.
This class of hypersurface normalisations parametrises abundant superintegrable systems. The associated abundant systems can, indeed, be obtained by the methods of \cite{KSV2024,KSV2024_bauhaus}, in particular by solving certain systems of partial differential equations (whose integrability is ensured).
This technique can also be used to study abundant systems with special geometric properties.
For instance, we investigate the special case of (local) graph immersions. In this situation, there exists a natural underlying Hessian structure, similar to \cite{AV2025}.
This allows us to reinterpret \emph{superintegrable structure functions} from the angle of hypersurface geometry. These functions were introduced in \cite{KSV2023} for abundant systems on constant curvature spaces and allow one to encode these systems in a single function.
\smallskip

The second main result of this paper concerns conformal rescalings.
In the theory of second-order superintegrable systems, conformal rescalings are well-established transformations \cite{KKMP2011}, which can be traced back to the classical Maupertuis-Jacobi transformation \cite{BKM1986,Maupertuis_1750,jacobi}.
In the context of affine hypersurface normalisations, on the other hand, conformal rescalings act by suitable changes of the transversal field and arise due to the fact that the property of being \emph{relative} does not determine the transversal field of an affine hypersurface normalisation completely \cite{Simon1988,nomizu-sasaki}.

We prove that our correspondence is compatible with these two concepts of conformal rescalings, allowing us to extend our correspondence consistently. We establish a correspondence between \emph{conformal classes} of abundant systems and their unique associated hypersurface immersions, without the need of fixing a transversal field.
This may bear significance for the link between abundant systems and hypergeometric orthogonal polynomials, which has been studied for conformal classes of abundant systems, cf.~\cite{KMP13}.
\smallskip

Targeting abundant superintegrable systems, this paper also contributes to a broader theory of second-order superintegrability. We remark that the superintegrable systems considered here admit a potential involving n+2 linear parameters \cite{KSV2023}. They are thus non-degenerate second-order maximally superintegrable systems in the sense of \cite{KKM-3}. For $n\geq4$ it is an open question whether all non-degenerate second-order maximally superintegrable systems are abundant \cite{KKM-3,KKM-4,KSV2023}. Our paper provides novel techniques for future research on this matter.
Second-order maximally superintegrable systems that do not satisfy the conditions of non-degeneracy have also been discussed in the literature, e.g.~\cite{BKM2020,EM2017,Vollmer2024,NV}. They are out of scope here and we leave them to future research, hypothesising that in general they do not admit an analogous realisation as affine hypersurface normalisations.
\medskip

The paper is organised as follows: to ensure a self-contained exposition, we begin with a brief review of affine hypersurface theory in Section \ref{sec:preliminaries}. It is based on the book \cite{SSV1991}, but the notation and organisation of the material has been considerably modified in order to ensure conciseness and to adapt it to our purposes. Readers already acquainted with the theory of affine hypersurfaces may skim through this section to quickly familiarise themselves with the notation.

In Section~\ref{sec:ACSIS.RACS}, we introduce abundant manifolds, which are based on the structural equations of abundant second-order superintegrable systems from \cite{KSV2024}, see Definition~\ref{defn:ACSIS}. We also define a special class of relative affine hypersurface normalisations, which we call \emph{abundant} (Definition~\ref{defn:RACS}). We then establish our main result, the correspondence between abundant hypersurface normalisations and abundant manifolds, for dimensions three and higher (Theorems~\ref{thm:main.1} and~\ref{thm:main.2}).

The main theme of Section~\ref{sec:conformal} is our discovery that the a priori unrelated concepts of conformal rescalings, on the one hand in the theory of 
superintegrable system and, on the other hand, in relative affine differential geometry, match precisely. 
Specifically, we show that our correspondence is compatible with conformal rescalings (Theorem~\ref{thm:main.3}), 
enabling us to define abundant hypersurfaces independently of a transversal field. 

In Section~\ref{sec:surfaces}, we proceed to studying the two-dimensional case. While the definition of abundant manifolds and abundant hypersurface normalisations is special in this case (Definitions~\ref{defn:ACSIS.2D} and~\ref{defn:RACS.2D}), their correspondence turns out to be analogous to higher dimensions (Theorem~\ref{thm:2D}).

The paper is concluded in Section~\ref{sec:ex} with applications and examples.
In particular, we discuss (local) graph immersions and abundant systems on spaces of constant sectional curvature.
On manifolds of dimension $n\geq3$ with constant sectional curvature, we find that abundant superintegrable systems (in the proper sense) can always be locally realised as graph immersions. We present an explicit counter-example in dimension two. We also formulate the conditions for a general abundant manifold to correspond locally to a graph immersion (Theorem~\ref{thm:graph.immersions}).

\section{Preliminaries}\label{sec:preliminaries}

In the usual hypersurface theory one considers an immersion $f:M\longto\R^{n+1}$ into the oriented Euclidean space, and equips $M$ with the induced metric and unit normal (determined by the orientations of $M$ and $\R^{n+1}$).
Its exterior curvature is described by the second fundamental form $I\!I\in\Gamma(\Sym^2T^*M)$ or, equivalently, the Weingarten tensor. The proper (i.e., orientation preserving) Euclidean transformation group acts on the set of such hypersurfaces by mapping the geometric structures (induced metric, unit normal and Weingarten tensor) into each other.

In affine hypersurface geometry, in contrast, one seeks geometric structures associated with hypersurface immersions $f:M\longto\R^{n+1}$ that are mapped into each other under affine transformations, which are more general.
We denote the standard (flat) connection on the ambient $\R^{n+1}$ by $\bar\nabla$.
In addition we fix a transversal field $\xi:M\longto\R^{n+1}$.
We call the tuple $(f,\xi)$ an \emph{affine hypersurface normalisation}.
For a given affine hypersurface normalisation $(f:M\longto\R^{n+1},\xi)$, the ambient connection on $\R^{n+1}$ decomposes as follows (for $X,Y\in\mathfrak X(M)$)
\begin{align*}
	\bar\nabla_{f_*(X)}f_*(Y) &= f_*(\nabla_XY)+G(X,Y)\xi\,,
	\\
	\bar\nabla_{f_*(X)}\xi &= -f_*(\hat A(X))+\Theta(X)\xi\,.
\end{align*}
In this way, $M$ becomes equipped with the (torsion-free) connection $\nabla$ induced by~$\bar\nabla$.
Moreover, $G(X,Y)$ defines a symmetric $(0,2)$-tensor field on $M$ called the \emph{affine fundamental form associated to $(f,\xi)$}. The tensor field $\hat A\in\Gamma(TM\otimes T^*M)$ is called the \emph{associated Weingarten endomorphism} (or Weingarten operator), and~$\Theta$ is called the \emph{1-form associated to~$(f,\xi)$}.

\begin{defn}~
	\begin{enumerate}[label=(\roman*)]
		\item An affine hypersurface normalisation $(f:M\longto\R^{n+1},\xi)$ is said to be \emph{non-degenerate} if $G\in\Gamma(T^*M\otimes T^*M)$ is non-degenerate.
		\item An affine hypersurface $f:M\longto\R^{n+1}$ is said to be \emph{non-degenerate} if it admits a non-degenerate affine hypersurface normalisation $(f,\xi)$.
	\end{enumerate}
\end{defn}
\begin{rmk}
	Let $f:M\longto\R^{n+1}$ be a non-degenerate affine hypersurface with non-degenerate affine hypersurface normalisation $(f,\xi)$. Let $(f,\xi')$ be another affine hypersurface normalisation of $f$.
	Then $(f,\xi')$ is also non-degenerate.
\end{rmk}

If the bilinear form $G$ is non-degenerate, it defines a pseudo-Riemannian metric on $M$, which is 
(positive or negative) definite if and only if the hypersurface is locally convex. We tacitly endow $M$ with this metric and call $G$ the \emph{metric associated to~$(f,\xi)$}. The bilinear form $A$ defined by $A(X,Y)=G(\hat A(X),Y)$ is similarly called the \emph{associated Weingarten tensor}.\smallskip

In order to keep the exposition concise, we are going to suppress the immersion~$f$ and thus abbreviate the structure equations as
\begin{subequations}\label{eq:structure}
	\begin{align}
		\bar\nabla_XY &= \nabla_XY+G(X,Y)\xi\,,
		\label{eq:structure.1}
		\\
		\bar\nabla_X\xi &= -\hat A(X)+\Theta(X)\xi\,.
		\label{eq:structure.2}
	\end{align}
\end{subequations}
By virtue of the metric, covariant and contravariant tensor arguments can be transformed one into the other. 
\smallskip

\paragraph{{\bf Notation.}}
Let $Q\in\Gamma(\Sym^pT^*M)$ be a symmetric $(0,p)$-tensor field. For a concise notation, we are going to denote by $\hat Q\in\Gamma(\Sym^{p-1}(T^*M)\otimes TM)$ the $(1,p-1)$-tensor field associated to $Q$.

\subsection{Non-degenerate hypersurfaces}
Consider an affine hypersurface immersion $f:M\longto\R^{n+1}$ with transversal field $\xi$.
From now on, we fix the standard volume form $\bar\omega$ on $\R^{n+1}$.
Without further mentioning, we are going to assume that $G$ in~\eqref{eq:structure} is \emph{non-degenerate}, and thus we endow $M$ with the metric $G$.
By construction, there exist two natural connections  on $M$, namely the induced connection $\nabla$ and the Levi-Civita connection $\nabla^G$ of $G$.
We introduce their difference tensor
\begin{equation}\label{eq:c_hat}
	\hat C(X,Y)=\nabla_XY-\nabla^G_XY\,.
\end{equation}
By virtue of $G$, we associate the $(0,3)$-tensor field $C$ to $\hat C$, with
\begin{equation}
	C(X,Y,Z)=G(\hat C(X,Y),Z)\,,
\end{equation}
following the previously introduced convention.
We call $C$ the \emph{cubic associated to $(f,\xi)$}.
Note that $\hat C$ is symmetric, $\hat C\in\Gamma(\Sym^2(T^*M)\otimes TM)$, and that $C$ is symmetric in its first two arguments.

As $M$ is endowed with an orientation $o$, we denote by $\omega^G$ the natural Riemannian volume form on $(M,G,o)$.
In addition, there exists a second natural volume form on $M$, namely that induced by the standard volume form $\bar\omega$ on $\R^{n+1}$,
\begin{equation}\label{eq:volume.form}
	\omega(X_1,\dots,X_n) := \bar\omega(X_1,\dots,X_n,\xi)\,.
\end{equation}
In \cite{SSV1991}, the following basic properties are proven.
\begin{lem}\label{lem:omegas}
	A non-degenerate hypersurface normalisation $(f,\xi)$ with induced connection~$\nabla$, associated metric $G$ and natural volume forms $\omega$ and $\omega^G$ satisfies
	\begin{align}
		(\nabla_ZG)(X,Y) &= -C(Z,X,Y)-C(Z,Y,X),
		\label{eq:nabla.G}
		\\
		\nabla_X\omega &= -\Theta(X)\omega,
		\label{eq:nabla.omega}
		\\
		\nabla_X\omega^G &=-(n+2)u(X)\,\omega^G,
		\label{eq:nabla.omegaG}
	\end{align}
	where
	\[ (n+2)u(X):=\tr(Y\longto\hat C(X,Y)) \]
	and where $\Theta$ denotes the 1-form associated to $(f,\xi)$.
\end{lem}

\subsection{Relative affine hypersurface normalisations}
Let $f:M\to\R^{n+1}$ be a (non-degenerate) hypersurface.
In the current subsection it is shown that there exists a transversal field~$\xi$, such that the 1-form associated to the normalisation $(f,\xi)$ satisfies $\Theta=0$.
\begin{defn}
	An affine hypersurface normalisation $(M,\xi)$ such that its associated 1-form satisfies $\Theta=0$ is called \emph{relative}.
\end{defn}
Using this terminology, given a hypersurface $f$, the following lemma shows that there always exists a normalization $(f,\xi)$ which is relative.
\begin{lem}\label{lem:relative.normalisation}
	Let $f:M\to\R^{n+1}$ be a non-degenerate hypersurface with (arbitrary) normalisation $(f,\xi)$, whose associated 1-form is $\Theta$.
	Then $(f,\xi'=\xi+\hat\Theta)$ is relative. 
\end{lem}
\begin{proof}
	For $(f,\xi)$, we denote the induced connection by $\nabla$, and its associated metric, cubic, 1-form and Weingarten operator by $G$, $C$, $\Theta$ and $\hat A$, respectively.
	Let $\xi':=\xi-w$, where $w\in\Gamma(TM)$. Then
	\begin{equation*}
		\bar\nabla_XY = \nabla_XY+G(X,Y)(\xi'+w) = \nabla'_XY+G(X,Y)\xi'\,.
	\end{equation*}
	We conclude $G'=G$ and $\nabla'_XY=\nabla_XY+G(X,Y)w$ for $X,Y\in\Gamma(TM)$.
	On the other hand,
	\begin{align*}
		\bar\nabla_X\xi' &= \bar\nabla_X(\xi-w) = -\hat A(X)+\Theta(X)\xi-\nabla_Xw-G(X,w)\xi \\
		&= -\big( \hat A(X)-\Theta(X)w+\nabla_Xw+G(X,w)w\big)~+~\big(\Theta(X)-G(X,w)\big)\xi'\,.
	\end{align*}
	We conclude that the Weingarten operator associated to $(f,\xi')$ is given by
	\[ \hat A'(X)=\hat A(X)-\Theta(X)w+\nabla_Xw+G(w,X)w\,, \]
	and that the associated 1-form of $(f,\xi')$ satisfies
	\[ \Theta'(X)=\Theta(X)-G(X,w)=\Theta (X) -G'(X,w)\,. \]
	Therefore, $ \hat\Theta'=\hat\Theta-w$ and, choosing $w=\hat\Theta$, we arrive at $\hat\Theta'=0$.
\end{proof}

The following lemma in proven in \cite{SSV1991}.
\begin{lem}
	Let $(f,\xi)$ be a (non-degenerate) relative normalisation.
	Its associated cubic~$C$ is totally symmetric.
\end{lem}
\begin{proof}
	Consider the flatness condition $\bar R(X,Y)Z=0$ for the Riemann curvature tensor of $\bar\nabla$.
	Using~\eqref{eq:structure} together with the condition $\Theta=0$ for the associated 1-form of $(f,\xi)$, and decomposing the resulting equation into its horizontal and transversal part, we arrive at
	\begin{align*}
		0 = (\nabla_XG)(Y,Z)-(\nabla_YG)(X,Z) = C(Z,Y,X)-C(Z,X,Y)\,.
	\end{align*}
	Already being symmetric in its first two arguments, $C$ hence is totally symmetric.
\end{proof}
For a relative normalisation $(f,\xi)$, the associated cubic thus has only one independent trace, $(n+2)u(X):=\tr(Y\to\hat C(X,Y))$, and we get the decomposition
\begin{equation}\label{eq:C2Uu}
	C(X,Y,Z) = U(X,Y,Z)+u(X)G(Y,Z)+u(Y)G(X,Z)+u(Z)G(X,Y),
\end{equation}
where $U$ is totally symmetric and trace-free with respect to~$G$.

\subsection{Conormal fields}

Let $(f:M\to\R^{n+1},\xi)$ be a non-degenerate affine hypersurface normalisation, $\xi\in\Gamma(TM)$. We recall that we have fixed a volume form~$\omega$ on $M$.
As a dual concept of affine hypersurface normalisations, we introduce \emph{co-normalisations} $(f,\Xi)$, where a section
$\Xi\in\Gamma((f^*T\R^{n+1}/TM)^*)$ replaces the transversal field.
We are going to see that for any normalisation $(f,\xi)$ there is a unique co-normalisation $(f,\Xi)$, whereas for a given co-normalisation there are many possible normalisations $(f,\xi)$. Co-normalisations are also the appropriate context for the study of conformal transformations of affine hypersurface (co-)normalisations.
The results of this subsection can be found in~\cite{SSV1991}.

\begin{defn}
	Let $f:M\to\R^{n+1}$ be an affine hypersurface.
	The line bundle 
	\[
		CM := (TM)^0 \subset f^*T^*\mathbb{R}^{n+1},
	\]
	where $(T_pM)^0\cong (T_{f(p)}\R^{n+1}/(f_*T_pM))^*$ denotes the annihilator in $T^*_{f(p)}\mathbb{R}^{n+1}$ of the hyperplane $f_*T_pM\subset T_{f(p)}\mathbb{R}^{n+1}$, is called the \emph{co-normal bundle of $M$}.

	A nowhere vanishing section $\Xi\in\Gamma(CM)$ is called a \emph{co-normal field of $M$} and the pair $(f,\Xi)$ is called an \emph{affine hypersurface co-normalisation}.
\end{defn}

We denote by $\langle-,-\rangle$ the natural pairing of vectors and co-vectors in $\R^{n+1}$.
Let $\Xi\in\Gamma(CM)$ and $X\in\Gamma(TM)$. Then
\begin{equation}\label{eq:co-normalisation.vertical}
	\la\Xi,X\ra=0\,.
\end{equation}
A transversal field $\xi\in\Gamma(f^*T\R^{n+1})$
satisfying the condition
\begin{equation}\label{eq:co-normalisation.horizontal}
	\la\Xi,\xi\ra=1
\end{equation}
is a transversal field of $M$. A transversal field $\xi$ satisfying \eqref{eq:co-normalisation.horizontal} is not uniquely determined by $\Xi$.
On the other hand, let $(f,\xi)$ by a hypersurface normalisation. Then the conditions~\eqref{eq:co-normalisation.vertical} and~\eqref{eq:co-normalisation.horizontal} determine $\Xi$ uniquely.
\begin{defn}
	Let $(f,\xi)$ be a non-degenerate hypersurface normalisation. Let $\Xi$ be the co-normal determined uniquely by~\eqref{eq:co-normalisation.vertical} and~\eqref{eq:co-normalisation.horizontal}. Then we call $(f,\Xi)$ the co-normalisation dual to $(f,\xi)$ 
	and $\Xi$ the \emph{co-normal dual} to $\xi$. A \emph{relative co-normalisation} is a co-norma\-li\-sation dual 
	to a relative normalisation. 
\end{defn}

\begin{prop}[\cite{SSV1991}]
	Let $(f,\xi)$ be a relative hypersurface normalisation, denote by 
	\[ \scrL_\xi \subset f^*T\mathbb{R}^{n+1} \] 
	the line bundle spanned by $\xi$ and let $\Xi\in\Gamma(CM)$ be the co-normal dual to $\xi$. Then (pointwise) 
	\[ \bar\nabla_X\Xi\in (\scrL_\xi)^0 \cong (f^*T\mathbb{R}^{n+1}/\scrL_\xi )^* = T^*M \]
	and
	\[ (\bar\nabla_X\Xi )(Y)=-G(X,Y) \]
	for $X,Y\in\Gamma(TM)$.
\end{prop}
\begin{proof}
	Using the structure equations~\eqref{eq:structure} we obtain 
	\[ (\bar\nabla_X\Xi )(\xi) = X(1)-\Xi(\bar\nabla_X\xi) = -\Theta(X)=0 \]
	and
	\[
		(\bar\nabla_X\Xi )(Y) = -\Xi(\bar\nabla_XY) = -G(X,Y)\,.
		\qedhere
	\]
\end{proof}

\begin{cor}\label{cor:xi<->Xi}
	For any transversal field $\xi$ of $M$, there exists a unique co-normal field $\Xi\in\Gamma(CM)$ such that
	\[ \langle\Xi,\xi\rangle=1\,. \]
	Conversely, for $\Xi\in\Gamma(CM)$ there is a unique transversal field $\xi$ of~$M$ with
	\[ \langle\xi,\Xi\rangle=1\,,\qquad \langle\xi,\bar\nabla_X\Xi\rangle=0\quad\forall X\in\Gamma(TM)\,. \]
\end{cor}

For a hypersurface normalisation $(f,\xi)$ we have found the structural equations~\eqref{eq:structure}.
Similarly, one is able to formulate a structural equation for a co-normali\-sation $(f,\Xi)$ dual to a relative normalisation
$(f,\xi)$.
Let $\Xi\in\Gamma( CM)$.
We introduce the dual connection
\begin{equation}\label{eq:dual_conn}
	\nabla^*_XY = \nabla^G_XY-\hat C(X,Y),
\end{equation}
where $\hat C$ is defined in \eqref{eq:c_hat}.
Then we have, see \cite{SSV1991},
\begin{subequations}\label{eq:co-structure}
\begin{align}
	\bar\nabla_X\bar\nabla_Y\Xi &= -A(X,Y)\Xi-G(\nabla^*_XY,- ),
	\\
	(\bar\nabla_X\Xi )(Y) &= -G(X,Y)\,.\label{eq:barnabla}
\end{align}
\end{subequations}
	Indeed,the second equation has already been determined and the other equation is obtained computing
	\begin{align*}
		\langle\bar\nabla_X\bar\nabla_Y\Xi,Z\rangle
		&= X(\bar\nabla_Y\Xi(Z))-\langle\bar\nabla_Y\Xi,\bar\nabla_XZ\rangle
		\\
		&=-X(G(Y,Z)) + G(Y, \nabla_XZ) = -G(\nabla^*_XY,Z)
		\end{align*}
		and 
		\begin{align*}
		\langle\bar\nabla_X\bar\nabla_Y\Xi,\xi\rangle
		&= -\langle\bar\nabla_Y\Xi,\bar\nabla_X\xi\rangle = -G(Y,\hat A X) = -A(X,Y)
	\end{align*}
	for $X,Y,Z\in\Gamma(TM)$.

We introduce the dual volume form $\omega^*$ via
\begin{equation}\label{eq:dual.volume.form}
	\omega^*(Y_1,\dots,Y_n)
	:= \bar\omega^*(GY_1,\ldots ,GY_n, \Xi ),
\end{equation}
where $\bar\omega^*$ denotes the standard dual volume form on $\R^{n+1}$, $Y_i$ are tangent to the hypersurface, $GY:=G(Y,- )$ compare 
Equation~\eqref{eq:barnabla}. Note that our conventions are such that $\omega$ and $\omega^*$ define the same orientation.

Analogously to~\eqref{eq:nabla.omega}, for a relative co-normalisation we have
\begin{equation}\label{eq:nabla.ast.omega.ast}
	\nabla^*_X\omega^* = 0
\end{equation}
because, for a basis $(Y_k)_{k=1,\dots,n}$ of $TM$, 
\label{page:nabla*omega*}
\begin{align*}
	(\nabla^*_X\omega^*)(Y_1,\ldots,Y_n)
	&= X\omega^*(Y_1,\dots,Y_n)-\omega^*(\nabla^*_XY_1,\ldots,Y_n)-\ldots-\omega^*(Y_1,\dots,\nabla^*_XY_n) \\
	&= X\bar\omega^*(GY_1,\dots,GY_n,\Xi)
		-\sum_{k=1}^n \bar\omega^*(GY_1,\ldots,G\nabla^*_XY_k,\ldots,GY_n,\Xi) \\
	&= X\bar\omega^*(GY_1,\dots,GY_n,\Xi)
		+\sum_{k=1}^n \bar\omega^*(GY_1,\ldots,\bar\nabla_X\bar\nabla_{Y_k}\Xi,\ldots,GY_n,\Xi) \\
	&= (\bar\nabla_X\bar\omega^*)(GY_1,\dots,GY_n,\Xi)+\bar\omega^*(GY_1,\ldots,GY_n,\bar\nabla_X\Xi)
	= 0\,,
\end{align*}
where we have used~\eqref{eq:co-structure}.

The following lemma highlights a close link between the $1$-form $u$, defined in Lemma~\ref{lem:omegas}, and the two natural volume forms of $(f,\xi)$ which becomes important when we later study the hypersurfaces associated to superintegrable systems.
The following lemma is proven in~\cite{SSV1991}.
\begin{lem}\label{lem:exactness.T}
	Let $(f:M\longto\R^{n+1},\xi)$ be a (non-degenerate) relative normalisation with associated volume forms $\omega$ and $\omega^*$.
	Then
	\begin{equation}\label{eq:exactness.T}
		u = \frac1{2(n+2)} d  \ln  \frac{\omega}{\omega^*} .
	\end{equation}
\end{lem}
\begin{proof}
	We have that
	\begin{align*}
		\omega(\nabla_XY_1,\dots,Y_n)+\dots
		+\omega(Y_1,\dots,\nabla_XY_n)
		&=\tr(\theta (X))\omega(Y_1,\dots,Y_n),
		\\
		\omega^*(\nabla^*_XY_1,\dots,Y_n)+\dots
		+\omega^*(Y_1,\dots,\nabla^*_XY_n)
		&=\tr(\theta^*(X))\omega^*(Y_1,\dots,Y_n),
	\end{align*}
	where $\theta$ (respectively $\theta^*$) is the connection form of $\nabla$ (respectively $\nabla^*$) in the local frame $(Y_i)$. 
	We therefore conclude, recalling~\eqref{eq:nabla.omega} and~\eqref{eq:nabla.ast.omega.ast},
	\begin{align*}
		X\ln\frac{\omega}{\omega^*}
		&= X\ln(\omega(Y_1,\dots,Y_n))-X\ln(\omega^*(Y_1,\dots,Y_n)) \\
		&= \frac{\nabla_X(\omega(Y_1,\dots,Y_n))}{\omega(Y_1,\dots,Y_n)}
			-\frac{\nabla^*_X(\omega^*(Y_1,\dots,Y_n))}{\omega^*(Y_1,\dots,Y_n)}
		 \\
		&= \tr(Y\mapsto(\nabla-\nabla^*)_XY)
		\\
		&= 2\,\tr(Y\to\hat C(X,Y))
		= 2(n+2)u(X)\,.
	\end{align*}
\end{proof}

\subsection{Blaschke normalisations}\label{sec:equiaffine}

We now consider a special class of relative normalisations.
\begin{defn}
	A relative hypersurface normalisation $(f,\xi)$ is called \emph{Blaschke} if
	 the volume forms $\omega$, $\omega^*$ satisfy $\omega=\omega^*$. More generally, a  relative hypersurface normalisation $(f,\xi)$ is called 
	 \emph{homothetically Blaschke} if $\frac{\omega}{\omega^*}$ is a positive constant. A relative hypersurface co-normalisation is called \emph{(homothetically) Blaschke} if the corresponding dual normalisation is (homothetically) Blaschke.
\end{defn}
Since the group $\Aff_\SL(\R^{n+1})$ of equiaffine transformations, that is affine transformations preserving the standard volume form $\bar \omega$, preserves the volume forms $\omega, \omega^*$, the class of Blaschke normalisations is invariant under this group. Note that given a homothetically Blaschke normalisation $(f,\xi )$ we 
can always rescale $\bar\omega$ by a constant factor such that $\omega = \omega^*$. 

By virtue of Lemma~\ref{lem:exactness.T}, we have, for any homothetically Blaschke normalisation, that
\[ u = \frac{1}{n+2}\tr \hat{C} = 0\,, \]
which is known as \emph{apolarity condition}.
A relative affine hypersurface therefore is homothetically Blaschke precisely if its cubic $C$ is trace-free, and due to~\eqref{eq:nabla.omegaG} this condition is equivalent to $\nabla\omega^G=0$.

At this point it is worthwhile to confront the three volume forms $\omega^G,\omega$ and $\omega^*$. Previously it was shown that the volume form $\omega$ is parallel with respect to $\nabla$ if and only if $\Theta= 0$, i.e.~precisely for relative hypersurfaces. The analogous result for $\omega^*$ and $\nabla^*$ follows after a review of the computation on page~\pageref{page:nabla*omega*}.
Likewise, for relative hypersurface normalisations, Equation~\eqref{eq:nabla.omegaG} together with \eqref{eq:dual_conn} guarantees that $\omega^G$ is parallel with respect to all three volume forms if and only if $u=0$, i.e.~precisely for normalisations which are homothetically Blaschke.
Later, in Section~\ref{sec:conformal.transformations.hypersurfaces}, we will find that, by virtue of a conformal transformation, a relative affine hypersurface normalisation can always be mapped onto a Blaschke normalisation, and that this has a natural correspondence on the level of conformally superintegrable systems.

\subsection{Existence and uniqueness theorems}

The structure equations~\eqref{eq:structure} for a relative affine hypersurface normalisation are
\begin{subequations}\label{eq:structure.relative}
	\begin{align}
		\bar\nabla_XY &= \nabla_XY+G(X,Y)\xi\,,
		\label{eq:structure.relative.1}
		\\
		\bar\nabla_X\xi &= -\hat A(X)\,.
		\label{eq:structure.relative.2}
	\end{align}
\end{subequations}
The goal in the present section is to characterise the data needed to uniquely determine a relative affine hypersurface normalisation. To this end the (first-order) integrability conditions for~\eqref{eq:structure.relative} are derived. Subsequently an existence and uniqueness theorem is proven, allowing one to immerse into $\R^{n+1}$ a manifold of dimension $n\geq2$, endowed with a pseudo-Riemannian metric $G$ and a totally symmetric cubic tensor $C$, as a relative affine hypersurface normalisation whose associated metric and cubic are $G$ and $C$, respectively.
\begin{prop}[cf.\ \cite{SSV1991},\cite{nomizu-sasaki}]\label{prop:integrability}
\begin{subequations}\label{eq:integrability}
	The first-order integrability conditions of~\eqref{eq:structure.relative.1} are
	\begin{align}
		R^G(X,Y)Z
		&= \hat C(Y,\hat C(X,Z))-\hat C(X,\hat C(Y,Z))
		\label{eq:integrability.gauss.1} \\
		&\quad +\frac12\,\big(A(Y,Z)X-A(X,Z)Y
		\nonumber \\
		&\qquad\qquad +\hat{A}(X)G(Y,Z)-\hat{A}(Y)G(X,Z)\big)\,,
		\nonumber
		\\
		(\nabla^G_X C)(Y,Z,W)&-(\nabla^G_Y C)(X,Z,W)
		\label{eq:integrability.gauss.2} \\
		&= \frac12\,\big(A(X,W)G(Y,Z)+A(X,Z)G(Y,W)
		\nonumber \\
		&\qquad -A(Y,W)G(X,Z)-A(Y,Z)G(X,W)\big)\,,
		\nonumber
		\\
		C(X,Z,Y) &= C(Y,Z,X)\,,
		\label{eq:integrability.C}
	\end{align}
	and the first-order integrability conditions of~\eqref{eq:structure.relative.2} are
	\begin{align}
		(\nabla^G_X\hat{A})(Y)-(\nabla^G_Y\hat{A})(X) &= \hat C(Y,\hat{A}(X))-\hat C(X,\hat{A}(Y))\,,
		\label{eq:integrability.codazzi}
		\\
		A(X,Y) &= A(Y,X)\,.
		\label{eq:integrability.dTheta}
	\end{align}
\end{subequations}
\end{prop}
Note that~\eqref{eq:integrability.C} implies that $C$ is totally symmetric. Likewise, we conclude from~\eqref{eq:integrability.dTheta} that~$\hat A$ is self-adjoint with respect to $G$, and from~\eqref{eq:integrability.codazzi} that $A$ is a Codazzi tensor with respect to~the induced connection $\nabla$.
\begin{proof}
	The last two conditions are obtained using
	\[ \bar R(X,Y)\xi = 0\,, \]
	which, written out and split into tangential and vertical part, is equivalent to
	\begin{equation*}
		\left[ \nabla_Y\hat A(X)-\nabla_X\hat A(Y) \right]
		+\left[ G(Y,\hat A(X))-G(X,\hat A(Y)) \right]\xi = 0\,.
	\end{equation*}
	This yields~\eqref{eq:integrability.codazzi} and~\eqref{eq:integrability.dTheta}, after rewriting $\nabla$ in terms of $\nabla^G$.
	The remaining three conditions are obtained, analogously, from
	\[ \bar R(X,Y)Z = 0\,,  \]
	which is written out as 
	\begin{multline*}
		\left[ R^\nabla(X,Y)Z-G(Y,Z)\hat{A}(X)+G(X,Z)\hat{A}(Y)\right]
		+\left[ (\nabla_XG)(Y,Z)-(\nabla_YG)(X,Z) \right]\xi = 0\,.
	\end{multline*}
	The equation obtained (with the help of \eqref{eq:nabla.G}) from the coefficient of $\xi$ is~\eqref{eq:integrability.C}.
	Now consider the equation
	\begin{equation}\label{eq:Ric.via.A}
		R^\nabla(X,Y)Z-G(Y,Z)\hat A(X)+G(X,Z)\hat A(Y)= 0\,.
	\end{equation}
	Rewritten in terms of $\nabla^G$, we have
	\begin{multline}\label{eq:RicG.via.A}
		R^G(X,Y)Z + (\nabla^G_X\hat C)(Y,Z)-(\nabla^G_Y\hat C)(X,Z) \\
		+\hat C(X,\hat C(Y,Z))-\hat C(Y,\hat C(X,Z)) 
		-G(Y,Z)\hat A(X)+G(X,Z)\hat A(Y)=0.
	\end{multline}
	Now we apply the composition of the following projectors to the $(0,4)$-tensor associated to the above curvature tensor via $G$:
	symmetrisation in the pair of components $(1,3)$, symmetrisation in $(2,4)$, 
	skew-symmetrisation in $(1,2)$ and skew-symmetrisation in $(3,4)$. 
	The vanishing of the resulting tensor is \eqref{eq:integrability.gauss.1}. The vanishing of remaining part of \eqref{eq:Ric.via.A} 
	yields ~\eqref{eq:integrability.gauss.2}. Note that since~\eqref{eq:integrability.gauss.1} lies in $S^2\Lambda^2T^*M$ and~\eqref{eq:integrability.gauss.2} in $\Lambda^2T^*M\otimes S^2T^*M$, \eqref{eq:Ric.via.A} is, indeed, satisfied if and only if \eqref{eq:integrability.gauss.1} and \eqref{eq:integrability.gauss.2} are.
\end{proof}

Recall the conjugate connection, which satisfies
\[ \nabla^*_XY = \nabla^G_XY-\hat C(X,Y) = \nabla_XY-2\hat C(X,Y)\,. \]
We denote its curvature tensor by $R^*$, whose Ricci tensor is $\Ric^*(X,Y):=\tr(Z\mapsto R^*(Z,X)Y)$.
\begin{lem}\label{lem:Ric.of.hypersurface.A}
	Let $(f,\xi)$ be a relative non-degenerate affine hypersurface normalisation.
	Then
	\begin{equation}\label{eq:Ric.ast.A}
		\Ric^*(X,Y) = (n-1)A(X,Y)\,.
	\end{equation}
\end{lem}
\begin{proof}
	The conjugate connection satisfies
	\[ XG(Y,Z)=G(\nabla_XY,Z)+G(Y,\nabla^*_XZ)\,, \]
	which implies
	\[ G(R^\nabla(X,Y)Z,W)+G(Z,R^*(X,Y)W) = 0\,. \]
	Using this latter identity and the integrability condition~\eqref{eq:Ric.via.A}, 
	we obtain
	\begin{equation}\label{eq:Riem.A}
		R^*(X,Y)Z = A(Y,Z)X-A(X,Z)Y\,.
	\end{equation}
	The claim then follows immediately by taking the trace.
\end{proof}
Note that~\eqref{eq:Ric.ast.A} implies that the Weingarten form is redundant in the relative case, i.e.~it is determined by $\nabla$ and $G$.

The following theorem is a fundamental existence and uniqueness result about non-de\-ge\-nerate affine hypersurface normalisations.
It is proven in~\cite[\S~4.9]{SSV1991} and \cite{Simon1988}. The proof we give below is similar to the one in~\cite[Ch.~II, Thm 8.1]{nomizu-sasaki}.
\smallskip

\pagebreak[2]
\begin{thm}\label{thm:existence.uniqueness}~
	\begin{enumerate}[label=(\roman*)]
		\item
		Let $M$ be a (connected) simply connected, oriented smooth manifold of dimension $n\geq2$. Let $G$ be a (pseudo-)Riemannian metric and $C$ a totally symmetric $(0,3)$-tensor field on $M$.
		Define an endomorphism $\hat A:TM\to TM$ by
		\[ (n-1)A(X,Y):=\Ric^*(X,Y).\]
		Assume that $C,G$ and $A$ satisfy the equations in Proposition~\ref{prop:integrability}.
		Then there is a hypersurface immersion $f:M\to\R^{n+1}$ and a transversal field $\xi$ defining a \emph{relative} hypersurface normalisation such that $\nabla=\nabla^G+\hat C$ is its induced connection and $G$ is its associated metric.
		\item
		Let $(f,\xi)$ and $(f',\xi')$ be two relative hypersurface normalisations on a connected manifold $M$.
		If their associated cubics $C$ resp.~$C'$ and metrics $G$ resp.~$G'$ satisfy
		\[ C'=C\quad\text{and}\quad G'=G\,, \]
		then $f$ and $f'$ are affinely equivalent.
	\end{enumerate}
\end{thm}

\begin{proof}~
	
	\noindent(i)
	We consider the vector bundle $TM \oplus \R$, where $\R$ denotes the trivial line bundle over~$M$. 
	We denote by $\xi$ the canonical (trivialising) section of $\R$ and by 
	\[\Phi :TM\to TM \oplus\R\] 
	the natural inclusion.	Then, for $X,Y\in \mathfrak{X}(M)$, the equations
	\begin{align*}
		\bar\nabla_{X} (\Phi(Y)) &= \Phi(\nabla_XY)+G(X,Y)\xi\,,
		\\
		\bar\nabla_{X}\xi &= -\Phi(\hat A(X))\,,
	\end{align*}
	define a flat connection $\bar\nabla$ on $TM\oplus \R$ under the hypothesis that the conditions~\eqref{eq:integrability} hold.
	Since $M$ is simply connected, the flat bundle  is trivial: $TM\oplus \R\cong \R^{n+1}$. For that reason, we can 
	consider $\Phi$ as a vector valued one-form on $M$, $\Phi : TM \to \R^{n+1}$.
	
	We find $(d^{\bar\nabla}\Phi )(X,Y)=(\bar\nabla_X\Phi )Y-(\bar\nabla_Y\Phi )X=0$ and thus, by the Poincar\'e lemma, $\Phi=d^{\bar\nabla}f$ for a vector-valued function $f:M\to\R^{n+1}$. The function $f$ turns out to be an immersion 
	(since $\Phi=d^{\bar\nabla}f$) with $\xi$ as relative normalisation and induced data $G$ and $C$.
	
	\noindent(ii) Let us first consider the case of a simply connected $M$. Then part (i) shows that  the immersion $f: M \to \R^{n+1}$ is uniquely determined by the homomorphism $\Phi : TM \to \R^{n+1}$ up to adding a constant vector, i.e.\ up to a translation. 
	Since the isomorphism $TM\oplus \R\cong \R^{n+1}$ of flat bundles (and therefore $\Phi$ as a map $TM \to \R^{n+1}$) is unique up to a constant linear transformation, we obtain the uniqueness of $f$ up to an affine transformation. This proves that $f'$ and $f$ are affinely equivalent if $M$ is simply connected. 
	
	The general case is reduced to the simply connected case by considering the universal covering $\pi : \widetilde{M} \to 
	M$. We consider the immersions $\tilde{f} = f \circ \pi : \widetilde{M} \to 
	 \R^{n+1}$ and $\tilde{f}' = f' \circ \pi : \widetilde{M} \to 
	 \R^{n+1}$ with the induced data $\pi^*G$ and $\pi^*C$. By the previous argument in the simply connected case,
	 the immersions  $\tilde{f}$ and $\tilde{f}'$ are related by an affine transformation $\phi$ of $\R^{n+1}$: 
	 $\tilde{f}' = \phi \circ \tilde{f}$, which implies $f'= \phi \circ f$. 
\end{proof}

\subsection{Relative spheres}

We consider (another) special class of relative hypersurface normalisations, so-called  \emph{relative spheres}, which, in particular, include affine spheres. While the concept of relative spheres is a technical one, see \cite[\S~7.2.6]{SSV1991}, it is going to be a useful terminology in the following discussion.
\begin{defn}
	Let $(f:M\longto\R^{n+1},\xi)$ be a relative hypersurface normalisation.
	\begin{enumerate}
		\item If $\xi=\mu(p)\,(f(p)-x_0)$ for a scalar function $\mu:M\to\R$, $p\in M$, then $(f,\xi)$ is called a \emph{proper relative sphere} with center $x_0\in\R^{n+1}$.
		\item If $\xi=c_0\in\R^{n+1}\setminus\{0\}$, then $(f,\xi)$ is called an \emph{improper relative sphere}.
	\end{enumerate}
\end{defn}
\noindent If $(f,\xi)$ is Blaschke and a relative sphere, then $(f,\xi)$ is called an \emph{affine sphere}.
The following lemma characterises a relative sphere in terms of the Weingarten operator and the curvature, respectively.
\begin{lem}[\S~7.2.3 of \cite{SSV1991}]
	\label{lem:relative.sphere.characterisation}
	For a relative affine hypersurface normalisation $(f:M\to\R^{n+1},\xi)$ the following are equivalent:
	\begin{enumerate}[label=(\roman*)]
		\item
		$(f,\xi)$ is a relative sphere.
		\item
		The Weingarten operator satisfies
		$ \hat A=-\mu\Id $
		for a scalar function~$\mu$ on $M$.
		\item
		$\displaystyle \Ric^\nabla=\Ric^* $.
		\item 
		$\nabla C$ is totally symmetric.
	\end{enumerate}
\end{lem}
\begin{proof}
	The equivalence (ii)$\Leftrightarrow$(iv) follows immediately from~\eqref{eq:integrability.gauss.2}, taking into account that 
	$\nabla^GC$ is totally symmetric if and only if $\nabla C$ is.
	We now show (i)$\Leftrightarrow$(ii). Indeed, for an improper relative sphere, condition (ii) is obvious. In the case of a proper relative sphere, we have for $X\in TM$
	\[
	\bar\nabla_{X}\xi 
	=\frac{d\mu (X)}{\mu}\xi +\mu\,X\,,
	\]
	and the claim then follows since, due to the relative normalisation, the coefficient of $\xi$ has to vanish, thus $d\mu=0$.
	Conversely, if (ii) holds, Equation~\eqref{eq:integrability.codazzi} implies $d\mu=0$. Hence~\eqref{eq:structure.relative} becomes $\bar\nabla_X\xi=\mu X$ with a constant $\mu$. This can be solved by direct integration, yielding $\xi=\mu f+c$ for a constant $c$.

	It remains to show (ii)$\Leftrightarrow$(iii). A contraction of~\eqref{eq:Ric.ast.A} and~\eqref{eq:Ric.via.A}, respectively, yields
	\begin{equation*}
		\Ric^* = (n-1)A
		\qquad\text{and}\qquad
		\Ric^\nabla = \tr(\hat A)G-A.
	\end{equation*}
	Given (ii), $\Ric^\nabla = (n-1)\mu G=\Ric^*$ is then obvious.
	Conversely, given (iii), we infer $\tr(\hat A)\Id-\hat A = (n-1)\hat A$, and then conclude that $n\hat A = \tr(\hat A)\Id =: -n\mu\Id$.
\end{proof}

In the special case that the position vector field in $\R^{n+1}$ is transversal along the immersion $f : M \to \R^{n+1}$, the normalisation $(f,\xi=-f)$ is called \emph{centroaffine normalisation}. Note that centroaffine normalisations are examples of relative spheres, since $\hat{A} = \mathds{1}$.
In \cite{Ferapontov2004}, affine hypersurfaces with a flat centroaffine metric are related to the equations of associativity of 2-dimensional topological field theory. In this regard, note that \cite{Vollmer2025} finds that abundant second-order superintegrable systems carry the structure of a Frobenius manifold in the sense of Manin, cf.~\cite{Manin1999,Manin1996}.

We mention another special case of hypersurface normalisations that will become relevant later.
\begin{defn}\label{defn:quadric}
	A relative hypersurface normalisation $(f:M\longto\R^{n+1},\xi)$ is said to be \emph{quadric-type} if and only if
	\begin{equation}\label{eq:quadric-type}
		U(X,Y,Z)=0\,,
	\end{equation}
	where $U$ is the trace-free part of $C$, cf.\ \eqref{eq:C2Uu}.
\end{defn}

One can show that the condition~\eqref{eq:quadric-type} is independent of the choice of norma\-li\-sation, cf.~\cite[\S~7]{SSV1991}. Definition~\ref{defn:quadric} hence extends to a definition of quadric-type relative hypersurfaces.
Moreover, one can prove that a relative hypersurface (normalisation) is quadric-type if and only if it lies on a (non-degenerate) quadric.

\section{Correspondence between abundant manifolds and abundant hypersurface normalisations}\label{sec:ACSIS.RACS}

While the previous sections mainly provided a review of established aspects of affine hypersurface geometry, the current section introduces two structures which then are shown to be closely interrelated: abundant manifolds and abundant hypersurface normalizations. 
The first of these structures is defined on a (pseudo-)Riemannian manifold and is characterised by a totally symmetric tensor field as well as a scalar function. It naturally arises in the study of second-order (maximally) conformally superintegrable systems, see Remark~\ref{rmk:ACSIS} below.
Note that, at present, we confine ourselves to manifolds of dimension $n\geq3$. The two-dimensional case is going to be discussed later, in Section~\ref{sec:surfaces}, as it is subject to several particularities.

For a concise formulation of the following definition, we shall denote the Schouten tensor of $g$ by 
 \[ \P^g = \frac{1}{n-2}\left( \Ric^g - \frac{\tr_g (\Ric^g)}{2(n-1)}\,g \right), \]
and the Kulkarni-Nomizu product of tensor fields $B_1,B_2\in\Gamma(\mathrm{Sym}^2T^*M)$ by
\begin{align*}
	(B_1\owedge B_2)(X,Y,Z,W)
	&= B_1(X,Z)B_2(Y,W) + B_1(Y,W)B_2(X,Z) \\
	&\quad - B_1(X,W)B_2(Y,Z) - B_1(Y,Z)B_2(X,W)
\end{align*}
for $X,Y,Z,W\in\mathfrak X(M)$.
We also introduce the symmetric $(0,2)$-tensor field $\mathscr{S}$,
\[
	\mathscr{S}(X,Y) = \tr (S_X S_Y), \quad X,Y\in \mathfrak{X}(M)\,,
\]
where
\[
	S_X := \hat{S}(X,\cdot)\in \Gamma (\mathrm{End}(TM)),\quad X\in \mathfrak{X}(M)\,,
\]
and the tensor field $S_1 \in \Gamma ((T^*M)^{\otimes 3}\otimes T^*M)$ by
\begin{multline}\label{eq:S1}
	S_1(X,Y,Z) := g S_XS_YZ + 3 S(X,Y, \cdot ) dt (Z) + S(X,Y,Z)dt\\ + 
	\left(\frac{4}{n-2} \mathscr{S}(Y,Z) -3 S(Y,Z, \mathrm{grad}_g\, t)\right) g(X, \cdot )
\end{multline}
for $X,Y,Z\in\X(M)$.

\begin{defn}\label{defn:ACSIS}~
	Let $(M,g)$ be a conformally flat (pseudo-)Riemannian (oriented) manifold of dimension $n\geq3$.
	Assume it is equipped with a totally symmetric and trace-free $(0,3)$-tensor field $S$ and a smooth function $t$.	
	We say that $(M,g,S,t)$ is an \emph{abundant manifold}, if
	\begin{subequations}\label{eq:ACSIS.conds}
		\begin{enumerate}[wide,label=(\roman*)]
			\item\label{item:Hess.t} the Hessian of $t$ satisfies
				\begin{equation}
					\label{eq:ACSIS.conds.DDt}
					(\nabla^g)^2 t
					=3\P^g
					+\frac13\left( dt^2-\frac12 |\mathrm{grad}\, t|_g^2 g \right) 
					+\frac{1}{3(n-2)}\left(
					\mathscr{S}+\frac{(n-6)\,|S|^2_gg}{2(n-1)(n+2)}
					\right)\,, 
				\end{equation}
			\item\label{item:DS} the covariant derivative of $S$ satisfies
				\begin{equation}\label{eq:ACSIS.conds.DS}
					\nabla^gS = \frac13\, \Pi_{\mathrm{Sym}_0^3} S_1\,,
				\end{equation}
				where $\Pi_{\mathrm{Sym}_0^3} : (T^*M)^{\otimes 3}\otimes T^*M \to  \mathrm{Sym}_0^3T^*M\otimes T^*M$ denotes 
				the natural projection (described below) onto the trace-free symmetric tensors (in the first three arguments),
			\item\label{item:Weyl} the $(0,4)$ curvature tensor $\Riem^{S} = gR^S$ of the (torsion-free) connection $\nabla^S:=\nabla^g-\hat S$ is given by
			\begin{equation}\label{eq:ACSIS.conds.Weyl}
				\Riem^{S} = \mathsf{P}^S\owedge g \,,
			\end{equation}
			with
			\[ \mathsf{P}^S = \frac{1}{n-2}\left( \Ric^S - \frac{\tr_g (\Ric^S)}{2(n-1)}\,g \right), \]
			where we denote the Ricci tensor of $\nabla^S$ by $\Ric^S$.
		\end{enumerate}
	\end{subequations}
\end{defn}
We call $(S,t)$ an \emph{abundant structure} on $(M,g)$, if $(M,g,S,t)$ is a abundant manifold.\medskip

The following three remarks clarify aspects of the definition. The first remark concerns the projector $\Pi_{\mathrm{Sym}^3_0}$. The second remark will comment on the criterion~\ref{item:Weyl} of the definition. In the third remark, the definition of abundant manifolds is going to be motivated via results from the theory of second-order (maximally conformally) superintegrable systems.

\begin{rmk}[the projector $\Pi_{\mathrm{Sym}^3_0}$]
	Denote by 
	\[ \Pi_{\mathrm{Sym}^3} : (T^*M)^{\otimes 3}\otimes T^*M \to  \mathrm{Sym}^3T^*M\otimes T^*M\] the natural projector onto the totally symmetric component (in the first three arguments), i.e.
	\begin{align*}
		(\Pi_{\mathrm{Sym}^3}\Phi )(X,Y,Z,W)
		&= \frac16\Big( \Phi(X,Y,Z,W) + \Phi(Y,Z,X,W) + \Phi(Z,X,Y,W)
		\\
		&\qquad + \Phi(Y,X,Z,W) + \Phi(X,Z,Y,W) + \Phi(Z,Y,X,W) \Big)\,.
	\end{align*}
	The projection $\Pi_{\mathrm{Sym}_0^3} : (T^*M)^{\otimes 3}\otimes T^*M \to  \mathrm{Sym}_0^3(T^*M)\otimes T^*M$ appearing in the criterion~\ref{item:DS} of Definition~\ref{defn:ACSIS}, i.e.~in Equation~\eqref{eq:ACSIS.conds.DS}, then is naturally given by
	\[
		\Pi_{\mathrm{Sym}_0^3}\Phi
		= \Pi_{\mathrm{Sym}^3}\Phi  - \frac{3}{n+2} \Pi_{\mathrm{Sym}^3} (g\otimes \phi )\,,
	\]
	where we introduce
	\[
	\phi (X,W) := \tr_g\left( (\Pi_{\mathrm{Sym}^3}\Phi) (\cdot,\cdot,X,W)\right)\,.
	\]
\end{rmk}

In Remark~\ref{rmk:ACSIS} we shall motivate Definition~\ref{defn:ACSIS} quoting a characterisation of so-called \emph{abundant} (second-order maximally conformally) superintegrable systems given in \cite{KSV2024,KSV2023}. In preparation, we comment on the criterion~\ref{item:Weyl} in Definition~\ref{defn:ACSIS}, which we have formulated differently from these references.

\begin{rmk}[the curvature criterion~\ref{item:Weyl}]\label{rmk:Weyl.projector}
	Writing $R^S$ in terms of $R^g$ and $S$, we obtain
	\begin{align*}
		R^S(X,Y)
		&= R^g(X,Y) - (d^{\nabla^g}\hat S)(X,Y) 
			+ [S_X,S_Y]\\
			&= R^g(X,Y) 
			+ [S_X,S_Y]\, , 
	\end{align*}
	due to \eqref{eq:ACSIS.conds.DS}.
	Furthermore, we find $\Ric^S = \Ric^g - \mathscr{S}$ using that $S$ is trace-free and then 
	\begin{align*}
		\P^S &= \P^g - \frac{1}{n-2}\left(
			  \mathscr{S}
				-\frac{\tr_g (\mathscr{S})}{2(n-1)}\,g
			\right) \, .
	\end{align*}
	Introducing the projector $\Pi_{\mathrm{Weyl}_0}:\mathrm{Sym}^2_0(T^*M)^{\otimes 2}\to (\mathrm{Sym}^2(\Lambda^2T^*M))_0$
	onto totally trace-free (algebraic Weyl) tensors by
	\begin{equation*}
		\Pi_{\mathrm{Weyl}_0} B:=
		B_1
		- \left( \mathscr{B} - \frac{\tr_g (\mathscr{B})}{2(n-1)}\,g\right)\owedge g\,,
	\end{equation*}
	where
	\begin{multline*}
		B_1(X,Y,Z,W) 
		:= \frac14\Big( B(X,Z,Y,W) - B(X,W,Y,Z) - B(Y,Z,X,W) + B(Y,W,X,Z) \Big)
	\end{multline*}
	and
	\[
		\mathscr{B}(X,Y) := \frac{1}{n-2}\tr_g (B_1(\cdot,X,\cdot,Y))\,.
	\]
	Note that the first Bianchi identity is satisfied for $\Pi_{\mathrm{Weyl}_0} B$ together with the skew symmetry in the first and the last pair of arguments, and the symmetry of these pairs.
	
	We therefore obtain that, given~\eqref{eq:ACSIS.conds.DS}, the conditions~\eqref{eq:ACSIS.conds.Weyl} and
	\begin{equation}\label{eq:ACSIS.conds.Weyl.alternative}
		\Pi_{\mathrm{Weyl}_0} \mathfrak{S} = 0, 
	\end{equation}
	are equivalent, where 
	\[ \mathfrak{S}(X,Y,Z,W) := g(S_XY, S_ZW)\]
	implying that the condition \eqref{eq:ACSIS.conds.Weyl} can be replaced by~\eqref{eq:ACSIS.conds.Weyl.alternative} in Definition~\ref{defn:ACSIS}.
\end{rmk}

The structure introduced in Definition~\ref{defn:ACSIS} naturally appears in the study of second-order superintegrable systems, which we are now going to briefly outline in the last remark of this series of comments.

\begin{rmk}[abundant structures]\label{rmk:ACSIS}
The equations in Definition~\ref{defn:ACSIS} are taken from \cite{KSV2024}, where they naturally arise for abundant second-order conformally (maximally) superintegrable systems \cite{KSV2023,KSV2024}.
Let $M$ be simply connected and let $\{-,-\}$ be the canonical Poisson bracket on $T^*M$ induced by the tautological 1-form; denote canonical (local) Darboux coordinates by $(\mathbf x,\mathbf p)=(x^1,\ldots,x^n,$ $p_1,\ldots ,p_n)$. We introduce the \emph{Hamiltonian} $H(\mathbf x,\mathbf p)=g^{-1}(\mathbf p,\mathbf p)+V(\mathbf x)$, where $V:M\to\R$ is a function on $M$ (called \emph{potential}). 
Up to a constant conventional factor, which is absorbed in $(g,V)$, this is the total energy of a standard mechanical system.  A function $F:T^*M\to\R$ is called \emph{(conformal) integral (of the motion)} for $H$ if
\begin{equation}\label{eq:integral}
	\{H,F\} = \rho_F H\,,
\end{equation}
where $\rho_F$ is a function on $T^*M$.
An integral is called \emph{second-order} (or \emph{quadratic}) if it is a quadratic polynomial in the fibre coordinates $\mathbf p$ of the form
\[
	F=K^{ij}(\mathbf x)p_ip_j+W(\mathbf x)\,
\]
(the Einstein convention applies).
For second-order integrals $F$, $\rho_F= -2\rho \circ g^{-1} : T^*M \to \R$ is given by a 1-form $\rho$ on $M$.
The property of being second-order is preserved under isometries of $(M,g)$ and hence under symplectomorphisms of $T^*M$ that preserve the tautological 1-form and the Hamiltonian \cite{Marsden&Ratiu_1999,daSilva_2004}.
A direct computation shows that \eqref{eq:integral} is equivalent to the condition that $K=K_{ij}dx^idx^j$ is a conformal Killing tensor\footnote{that is $(\nabla^g_X K)(X,X) = \rho (X) g(X,X)$ for all $X\in\X(M)$.}, where $K_{ij}=g_{ia}g_{jb}K^{ab}$, together with the condition
\[
	dW = K( \mathrm{grad}\, V)-\rho V \,,
\] 
where $K : \Gamma(TM) \to \Gamma(T^*M)$, $X \mapsto K(X, \cdot )$.
Differentiating yields the \emph{Bertrand-Darboux equation} 
\begin{equation}\label{eq:Bertrand-Darboux}
	d(K( \mathrm{grad}\, V))=d(\rho V)\, , 
\end{equation}%
a compatibility condition for $(K,V)$, which is sufficient to determine $F$ up to an integration constant if the pair $(K,V)$ is known.  
Consider a subspace $\mathcal K$ in the space of conformal Killing tensors of $g$, and a subspace $\mathcal V$ in the space of functions on $M$, such that
\begin{enumerate}[label=(\roman*)]
	\item Equation~\eqref{eq:Bertrand-Darboux} holds true for all $(K,V) \in\mathcal K\times \mathcal V$, 
	\item the subspace $\{ \hat K_x \mid K \in \mathcal K\} \subset \mathrm{End} \, T_xM$ is irreducible (i.e.\ generates an irreducible
	subalgebra of $\mathfrak{gl}(T_xM)$)  
	for all $x\in M$, where $\hat K_x$ denotes
	the endomorphism $g_x^{-1} \circ K_x$,
	\item $\dim(\mathcal K)=\frac{n(n+1)}{2}-1$ and $\dim(\mathcal V)=n+2$, 
	\item any $K\in\mathcal K$ is trace-free.
\end{enumerate}
An \emph{abundant} second-order (conformally maximally) superintegrable system is given by $(M,g,\mathcal K,\mathcal V)$ if there are $K^{(\alpha)}\in\mathcal K$, $1\leq\alpha\leq 2n-2$, with associated integrals $F^{(\alpha)}$ such that $(F^{(\alpha )})_{0\leq\alpha\leq 2n-2}$ are functionally independent (where we denote $F^{(0)}=H$).

In \cite{KSV2023,KSV2024} it is shown that an \emph{abundant} second-order (conformally maximally) superintegrable system gives rise to a unique symmetric and trace-free tensor field $S$ and a scalar function $t$ (unique up to addition of a constant) via
\begin{multline*}
	(\nabla^g)^2V(X,Y)-\frac{1}{n}(\Delta_g V)g(X,Y)\\ 
	= 
	X(t)g(Y,\grad_gV)+Y(t)g(X,\grad_gV)-\frac2n g(X,Y)\,dt(\grad_gV)\\
	+ S(X,Y,\grad_gV)+\tau(X,Y)V
	\,,
\end{multline*}
$X,Y\in\X(M)$, where
\begin{multline}\label{eq:ACSIS.conds.Ric-Tau}
	\tau(X,Y)
	= \frac{2}{3(n-2)}\mathring{\mathscr{S}}(X,Y)
			- \frac{2}{3}\,\Big[ dt(\hat{S}(X,Y)) + X(t)Y(t) \Big]_\circ
	+ 2\mathring{\P}^g(X,Y)\,.
\end{multline}
We emphasise that $S$ and $t$ are determined by the space of conformal Killing tensors associated with the system.
We also note that the conformal Killing tensors $K\in\mathcal K$ satisfy a prolongation system of the form $\nabla K=P(K)$, where $P\in\Gamma(\Sym^2_0(T^*M)\otimes T^*M\otimes \Sym^2(TM))$, which is considerably simpler than generally for conformal Killing tensors \cite{GL2019,Weir1977,Wolf1998}.
We use a diacritic $\circ$ to denote the trace-free part, i.e.
\[
	\mathring{\mathscr S} = \mathscr{S} - \frac1n\,\tr(\hat{\mathscr{S}})\,g\,,
\]
and $\mathring{\P}^g=\P^g-\frac1n\,J\,g$, where $J$ denotes the trace of $\P^g$. We shall use this notation analogously for other tensor fields later whenever there is no risk of confusion about the metric (or, more generally, conformal structure) used.

It is shown in  \cite{KSV2024} that $(M,g,S,t)$ satisfies the conditions of an abundant manifold.
Moreover, it is shown there that the system of PDEs composed of~\eqref{eq:ACSIS.conds.DDt} and~\eqref{eq:ACSIS.conds.DS} can be integrated around $x_0\in M$ (up to conformal rescalings) for given initial data $S_{ijk}(x_0)$ satisfying the algebraic condition~\eqref{eq:ACSIS.conds.Weyl.alternative}.
\end{rmk}

The second structure introduced in the current section is a special class of relative affine hypersurface (co-)normalisations.
We shall show later, in Sections~\ref{sec:proof.acsis2racs} and~\ref{sec:proof.racs2acsis} that abundant manifolds on simply connected spaces of dimension $n\geq3$ are in 1-to-1 correspondence to such special affine hypersurface (co-)normalisations.

For a concise notation we introduce the symmetric $(0,2)$-tensor field $\mathscr{U}$,
\[
	\mathscr{U}(X,Y) = \tr (U_X U_Y), \quad X,Y\in \mathfrak{X}(M)\,,
\]
where $U_X := \hat{U}(X,\cdot)\in \Gamma (\mathrm{End}(TM))\ \ \forall X\in \mathfrak{X}(M)$, compare \eqref{eq:C2Uu}.
As before, we denote the metric and the (dual) connection associated to a relative hypersurface normalisation by $G$ and $\nabla$ (respectively $\nabla^*$), and we denote the Weingarten operator by $\hat A$. We recall that a diacritic hat $\hat{~}$ denotes the $(1,p)$-tensor field associated, using $G$, to a symmetric $(0,p+1)$-tensor field.
We denote by $\Pi_{\mathrm{Sym}^m}:(T^*M)^{\otimes m}\to\mathrm{Sym}^m(T^*M)$ the natural projector onto the totally symmetric component.
Moreover, we denote by $\tr^\mathrm{Sym}_G:(T^*M)^{\otimes r}\to\mathrm{Sym}^{r-2}(T^*M)$ the ``symmetrised'' trace with respect to $G$,
\[
	\tr^\mathrm{Sym}_G(B) := \tr_G(\Pi_{\mathrm{Sym}^r}B)\,.
\]
Finally, we introduce, for $m\geq2$, the natural projector $\Pi_{\mathrm{Sym}^m_{0,G}}:(T^*M)^{\otimes m}\to\mathrm{Sym}^m_0(T^*M)$ onto the totally symmetric and (with respect to $G$) trace-free component by
\begin{align*}
	m&=2: &
	\Pi_{\mathrm{Sym}^2_{0,G}} B_2 &= \Pi_{\mathrm{Sym}^2}\left[ B_2-\frac12\,\tr^\mathrm{Sym}_G(B_2)g \right]\,,
	\\
	m&=3: &
	\Pi_{\mathrm{Sym}^3_{0,G}} B_3 &= \Pi_{\mathrm{Sym}^3}\left[ B_3-\frac{3}{n+2}\,\tr^\mathrm{Sym}_G(B_3)\otimes g \right]\,,
	\\
	m&=4: & 
	\Pi_{\mathrm{Sym}^4_{0,G}} B_4 &= \Pi_{\mathrm{Sym}^4}\left[
			B_4-\frac{6}{n+4}\,\left( \tr^\mathrm{Sym}_G(B_4) -\frac{ \tr_G(\tr^\mathrm{Sym}_G(B_4)) \otimes g }{2(n+2)} \right) \otimes g
		\right],
\end{align*}
and so forth, where $B_k\in(T^*M)^{\otimes k}$ (we shall only need the cases specified above).
We usually write $\Pi_{\mathrm{Sym}^m_0}=\Pi_{\mathrm{Sym}^m_{0,G}}$ if there is no risk of confusion about the underlying metric.

\begin{defn}\label{defn:RACS}
	Let $f:M^n\to\R^{n+1}$ be a relative (non-degenerate) affine hypersurface with conormal field $\Xi$, $n\geq3$. We denote its connection by $\nabla$, its metric by $G$ and its Weingarten tensor by $A$.
	Assume that the manifold $(M,G)$ is Weyl flat.
	Then we say that $(f,\Xi)$ is an \emph{abundant hypersurface co-normalisation} if the conditions
	\begin{subequations}\label{eq:RACS.conditions}
		\begin{align}
			\Pi_{\mathrm{Sym}^4_0}(\nabla^GU)
			&= \Pi_{\mathrm{Sym}^4_0}\left[
					\mathfrak U + 4\,U\otimes u
				\right]\,,
			\label{eq:RACS.conditions.DU}
			\\
			\nabla^G_Xu
			&= \P^G +u\otimes u -\frac{\lvert u\rvert^2}{2}\,\,G
					+\frac{1}{n-2}\left( \mathscr{U} +\frac{n-6}{2(n+2)(n-1)}\,\lvert U\rvert^2\, G \right)
			\label{eq:RACS.conditions.Du}
		\end{align}
	\end{subequations}
	are satisfied, where
	\[
		\mathfrak{U}(X,Y,Z,W) := g(U_XU_YW,Z)
	\]
	for $X,Y,Z,W\in\X(M)$.
	Analogously, we say that an affine hypersurface normalisation $(f,\xi)$ is an \emph{abundant hypersurface normalisation} if its associated co-normalisation is.
\end{defn}
\noindent Note that the conditions~\eqref{eq:RACS.conditions} have to hold in addition to the hypersurface equations~\eqref{eq:structure.relative} and~\eqref{eq:integrability}.

The proof of the following theorem is going to be given in Section~\ref{sec:proof.acsis2racs}.
\begin{thm}\label{thm:main.1}
	Let $(M^n,g,S,t)$ be a simply connected abundant manifold of dimension $n\geq3$.
	Then there is an immersion $f:M\to\R^{n+1}$ (unique up to affine transformations of the ambient space) with relative co-normalisation $\Xi$, such that $(f,\Xi)$ is an abundant hypersurface co-normalisation.
\end{thm}

We also show a converse statement: Given a relative affine hypersurface \mbox{(co-)}norma\-li\-sation that satisfies the `superintegrability conditions' \eqref{eq:RACS.conditions}, we obtain a system satisfying the structural equations~\eqref{eq:ACSIS.conds}.
Together with the results in \cite{KSV2023,KSV2024}, this provides a method for obtaining conformally superintegrable systems from relative affine hypersurfaces.

The proof of the following theorem is going to be given in Section~\ref{sec:proof.racs2acsis}.
\begin{thm}\label{thm:main.2}
	Let $M^n$, $n\geq3$, be simply connected and assume that $(f:M^n\to\R^{n+1},\Xi)$ is an abundant hypersurface co-normalisation with associated connection $\nabla$ and metric $G$.
	Consider the (totally symmetric) cubic
	\[
		C = -\frac12\nabla G\,,
	\]
	and define
	\[
		u := \frac1{n+2}\tr_G(C)\,,\qquad U:=C-3\Pi_\mathrm{Sym^3} u\otimes g\,,
	\]
	as in~\eqref{eq:C2Uu}.
	Then there is a function~$t$ on $M$ such that $(M,G,3U,t)$ is an abundant manifold and $3u = dt$.
\end{thm}
\begin{rmk}
	Note that $u$ is exact due to Lemma~\ref{lem:exactness.T}. 
	This lemma also allows us to refine the correspondence established by Theorems~\ref{thm:main.1} and~\ref{thm:main.2} by agreeing on the convention
	\[ t = \frac{1}{2(n+2)}\,\ln\frac{\omega}{\omega^*}\,, \]
	which satisfies $dt\coloneq 3u$, cf.\ Lemma~\ref{lem:exactness.T}. According to Theorem~\ref{thm:existence.uniqueness}, the affine hypersurface normalisation remains unchanged, if we add any constant to $t$. This corresponds with the freedom to rescale $\bar\omega$, which we discussed in Section~\ref{sec:equiaffine}. Note that our above convention ensures that $t=0$ precisely if $\omega=\omega^*$, i.e.\ in the case of a Blaschke immersion.
	To prove Theorem~\ref{thm:main.2}, it remains to prove that $(M,G,3U,t)$ satisfies the definition of an abundant manifold.
	This is done in Section~\ref{sec:proof.racs2acsis}.
\end{rmk}
Thanks to Theorems~\ref{thm:main.1} and~\ref{thm:main.2}, and because of Corollary~\ref{cor:xi<->Xi}, we therefore have the natural correspondences, for simply connected manifolds $M$,
\[
	(M^n,g,S,t)\stackrel{*}{\leftrightarrow} (f:M\longto\R^{n+1},\Xi)\leftrightarrow (f:M\longto\R^{n+1},\xi)\,,
\]
between abundant manifolds, on the one hand, and abundant co-normalisations and normalisations on the other.
Recall that the function $t$ in Theorem~\ref{thm:main.2}, for a given abundant hypersurface (co-)nor\-ma\-li\-sa\-tion, is determined only up to an (irrelevant) constant. On the level of superintegrable systems, this corresponds to multiplying the Hamiltonian by a constant, which does not change the dynamics of the system.
We thus introduce a concept of equivalence for abundant manifolds, which makes the correspondence at $\ast$ unique:
\begin{defn}\label{defn:ACSIS.equivalence}
	Let $M^n$ be a simply connected manifold, $n\geq3$.
	We say that two abundant manifolds $(M,g,S,t')$ and $(M,g,S,t)$ are \emph{equivalent} if $dt'=dt$.
\end{defn}

This implies that classifying abundant manifolds is not harder than classifying abundant hypersurface normalisations, and vice versa.
Hence, with the results in \cite{KSV2023,KSV2024}, it implies that classifying abundant second-order (maximally) conformally superintegrable systems is not significantly harder than classifying abundant hypersurface normalisations.

\subsection{Proof of Theorem~\ref{thm:main.1}}\label{sec:proof.acsis2racs}
In this section we employ the conditions~\eqref{eq:ACSIS.conds} seeking a suitable hypersurface normalisation, which will prove Theorem~\ref{thm:main.1}.
We proceed as follows: we make an ansatz for the cubic $C$ of a relative affine hypersurface and for its metric $G$, inducing an ansatz for the connection $\nabla$. We then check that the affine hypersurface conditions are satisfied for the ansatz. Theorem~\ref{thm:existence.uniqueness} subsequently ensures the existence of an immersion $f:M\to\R^{n+1}$ and of a hypersurface co-normalisation $\Xi$.

\subsubsection{Hypersurface Ansatz}
The abundant manifold given by the hypothesis provides us with the manifold $M$, the metric $g$, the totally symmetric and trace-free $(0,3)$-tensor field $S$ and the function $t$ on $M$.

Now consider the conditions~\eqref{eq:ACSIS.conds}.
As we aim at a relative \emph{non-degenerate} affine hypersurface, the obvious choice for the metric is
\begin{equation}\label{eq:ansatz.G}
	G:=g\,.
\end{equation}
By the hypothesis, $G$ therefore has vanishing (conformal) Weyl tensor, i.e.~$W^G=0$. Its Ricci tensor is determined by $S$ and $t$ via~\eqref{eq:ACSIS.conds.DDt}.
For the (totally symmetric) cubic field~$C$ our ansatz is
\begin{equation}\label{eq:ansatz.C}
	C(X,Y,Z) = \frac13\,\Big[ S(X,Y,Z)+X(t)G(Y,Z)+Y(t)G(X,Z)+Z(t)G(X,Y) \Big]\,.
\end{equation}
Using~\eqref{eq:Ric.ast.A}, we also agree upon (see Lemma~\ref{lem:gauss.conditions} below)
\begin{equation}\label{eq:ansatz.A}
	A =  \frac2n\left[
				(n+2)\nabla^G u - \div_G(C)
			\right]
			+\frac{1}{n(n-1)}\left( \Scal^G - \lvert C\rvert^2 + (n+2)^2\,\lvert u\rvert^2 \right)\,G\,,
\end{equation}
where we introduce $u:=\frac1{n+2}\tr_G(C)=\frac{1}{3}\,dt$ analogously to~\eqref{eq:C2Uu}, as well as
\[
	\mathscr{C}(X,Y)=\tr(C_XC_Y)\,,
\]
with $C_X=\hat C(X,\cdot)$ for $X,Y\in\X(M)$.
The formula~\eqref{eq:ansatz.A} can also be found in \S4 of \cite{SSV1991}.
Finally, we set
\[
	\mathfrak{C}(X,Y,Z,W) := g(C_XY,C_ZW)
\]
for $X,Y,Z,W\in\X(M)$.
\begin{lem}\label{lem:gauss.conditions}
	Equations~\eqref{eq:integrability.gauss.1} and~\eqref{eq:integrability.gauss.2} are equivalent to
	\begin{subequations}\label{eq:gauss.conditions}
	\begin{align}
		\mathring{A} &= \frac2n\left( (n+2)\nabla^Gu - \div_G(C) \right)\,,
		\label{eq:A0}
		\\
		(n-1)\tr(\hat A) &= \Scal^G-\lvert C\rvert^2 + (n+2)^2\,\lvert u\rvert^2\,,
		\label{eq:trA}
		\\
		\frac{n-2}{2}\mathring{A} + \frac{n-1}{n}\tr(\hat A)\,G
		&= \Ric^G-\mathscr{C} + (n+2)C(\hat u)\,,
		\label{eq:A.identity.Ric}
		\\
		\mathrm{Weyl}^G &= 2\Pi_{\mathrm{Weyl}_0} \mathfrak C = 2\Pi_{\mathrm{Weyl}_0} \mathfrak U\,,
		\label{eq:A.identity.Weyl}
		\\
		\Pi_{\mathrm{Codazzi}_0} \nabla^G U &= 0\,,
		\label{eq:A.identity.Codazzi}
	\end{align}
	\end{subequations}
	where $\mathring{A}=\Pi_{\mathrm{Sym}^2_0}A$ and where $\Pi_{\mathrm{Codazzi}_0}=\Pi_{\mathrm{Codazzi}_{0,G}}:T^*M\otimes\mathrm{Sym}^3_0(T^*M)\to T^*M\otimes\mathrm{Sym}^3_0(T^*M)$ is given by
	\begin{align*}
		2\Pi_{\mathrm{Codazzi}_0} B(X;Y,Z,W)
		&:= B(X;Y,Z,W)-B(Y;X,Z,W)
		\\
		&\quad +\frac1n\Big( b(X,Z)G(Y,W) + b(X,W)G(Y,Z)
		\\
		&\quad\qquad - b(Y,Z)G(X,W) - b(Y,W)G(X,Z) \Big)
	\end{align*}
	with $b=\tr_G(B)\in\Gamma(\mathrm{Sym}^2_0(T^*M))$ denoting the non-vanishing trace of $B$.
\end{lem}
\begin{proof}
	Equations~\eqref{eq:A.identity.Weyl} and~\eqref{eq:A.identity.Codazzi} follow by direct application of the respective projector to~\eqref{eq:integrability.gauss.1} and~\eqref{eq:integrability.gauss.2}, respectively (recall Remark~\ref{rmk:Weyl.projector}).
	Similarly, \eqref{eq:A0} is the trace of~\eqref{eq:integrability.gauss.2}, and \eqref{eq:A.identity.Ric} and~\eqref{eq:trA} are the simple and double trace of~\eqref{eq:integrability.gauss.1}.
	
	The converse direction follows immediately as~\eqref{eq:A0} and~\eqref{eq:A.identity.Codazzi} imply~\eqref{eq:integrability.gauss.2}, and \eqref{eq:A.identity.Weyl}, \eqref{eq:A.identity.Ric} and \eqref{eq:trA} imply~\eqref{eq:integrability.gauss.1}.
\end{proof}

Note that~\eqref{eq:gauss.conditions} imply
\begin{equation}\label{eq:Ric*A.renewed}
	R^*(X,Y)Z=A(Y,Z)X-A(X,Z)Y\,,
\end{equation}
compare Lemma~\ref{lem:Ric.of.hypersurface.A}, for the connection $\nabla^*=\nabla^G-\hat C$.

We moreover remark that the identity~\eqref{eq:trA} appears in \cite[\S4.12.2]{SSV1991} in a slightly different form interpreting $\tr(\hat A)$ as the relative mean curvature, and is called the \emph{theorema egregium of relative differential geometry}.

Note that the ansatz~\eqref{eq:ansatz.A} for $A$ immediately accounts for~\eqref{eq:A0} and~\eqref{eq:trA}.
Substituting this ansatz for $A$, together with the ansatz~\eqref{eq:ansatz.C} and the identification~\eqref{eq:ansatz.G}, into the integrability conditions, we will confirm that~\eqref{eq:integrability} are satisfied. Before carrying this out, we list some useful identities for the decomposition~\eqref{eq:C2Uu}.
\begin{lem}\label{lem:useful.C2Uu.identities}
	The following identities hold for a symmetric cubic $C$ of the form~\eqref{eq:C2Uu}, with trace-free component $U$ and trace $(n+2)u$:
	\begin{align*}
		\mathring{\mathscr{C}}(X,Y)
		&= \mathring{\mathscr{U}}(X,Y)+4U(X,Y,\hat u)+(n+6)\Big( u(X)u(Y)-\frac1nG(X,Y)|u|^2_G \Big)\,,
		\\
		\div_G(C)(X,Y) &= \div_G(U)(X,Y)+\nabla^G_Yu(X)+\nabla^G_Xu(Y)+\div_G(u)G(X,Y)\,.
	\end{align*}
\end{lem}
\begin{proof}
	The identities follow by a straightforward computation using~\eqref{eq:C2Uu}.
\end{proof}

\begin{lem}
	For $G,C$ and $A$ as in \eqref{eq:ansatz.G}, \eqref{eq:ansatz.C} and~\eqref{eq:ansatz.A}, the conditions~\eqref{eq:integrability} are satisfied.
\end{lem}
\begin{proof}
	Obviously, the symmetry properties~\eqref{eq:integrability.C} and~\eqref{eq:integrability.dTheta} are automatically satisfied by the ansatz. We now verify the remaining conditions of~\eqref{eq:integrability}.
	Due to Lemma~\ref{lem:gauss.conditions}, this is equivalent to checking~\eqref{eq:integrability.codazzi} and~\eqref{eq:gauss.conditions}.
	
	We begin with the equations~\eqref{eq:gauss.conditions}.
	The conditions~\eqref{eq:A0} and~\eqref{eq:trA} are, of course, already implemented in~\eqref{eq:ansatz.A}, and hence satisfied.
	Moreover, Equation~\eqref{eq:A.identity.Weyl} immediately follows from~\eqref{eq:ACSIS.conds.Weyl.alternative}, i.e.~from \eqref{eq:ACSIS.conds.Weyl}, and~\eqref{eq:A.identity.Codazzi} immediately follows from~\eqref{eq:ACSIS.conds.DS}.
	
	In order to continue, we agree upon the identifications, cf.~\eqref{eq:C2Uu},
	\[
		3U = S\,,\qquad 3u=dt\,.
	\]
	Inserting~\eqref{eq:ansatz.A} into~\eqref{eq:A.identity.Ric}, we obtain
	\begin{multline}\label{eq:A.comparison.divC}
		\frac{n-2}{n}\left( (n+2)\nabla^Gu - \div_G(C) \right)
		+\frac{1}{n}\left( \Scal^G-\lvert C\rvert^2 + (n+2)^2\,\lvert u\rvert^2 \right) G
		\\
		= \Ric^G -\mathscr{C}+(n+2)C({\hat u})
		\,.
	\end{multline}
	Due to Lemma~\ref{lem:useful.C2Uu.identities} we have
	\begin{equation*}
		\div_G(U) -n\nabla^Gu + \div_G(u)G
		= -n\mathring{\P}^G+\frac{n}{n-2}\mathring{\mathscr{U}} -n U(\hat u) 
		-n\left(  u\otimes u-\frac1n\,|u|^2\,G  \right),
	\end{equation*}
	and by contraction of~\eqref{eq:ACSIS.conds.DS}, we obtain
	\begin{equation}\label{eq:div(U).formula}
		U(X,Y,\hat u) = \tfrac{2}{n-2}\,\mathring{\mathscr{U}}(X,Y) -\frac1n\div_G(U)(X,Y)\,.
	\end{equation}
	With these identities, we return to~\eqref{eq:A.comparison.divC} and eliminate $U(X,Y,\hat u)$, concluding
	\begin{align*}
		\nabla^Gu -\frac1n\,\div_G(u)\,G
		&= \mathring{\P}^G+\frac1{n-2}\mathring{\mathscr{U}}+\left( u\otimes u -\frac1n\,\lvert u\rvert^2\,G \right)\,,
	\end{align*}
	which is satisfied because of~\eqref{eq:ACSIS.conds.DDt}.
	We have therefore verified~\eqref{eq:gauss.conditions}, and hence all of the conditions~\eqref{eq:integrability} except~\eqref{eq:integrability.codazzi}.
	
	In order to complete the proof, recall that~\eqref{eq:integrability.codazzi} is equivalent to
	\[
		 \nabla_Y\hat A(X)-\nabla_X\hat A(Y) = 0\,,
	\]
	cf.\ the proof of Proposition~\ref{prop:integrability}.
	We will now prove that $\nabla^*$ is projectively flat (see also Remark~\ref{rmk:proj.flatness}), from which this desired statement will follow.
	Since $n\geq3$, we need to confirm that
	\begin{equation}\label{eq:remains.to.prove.this}
		R^*(X,Y)Z-\frac1{n-1}\left[\Ric^*(Y,Z)X-\Ric^*(X,Z)Y\right]=0\,.
	\end{equation}
	Since we have already confirmed~\eqref{eq:gauss.conditions}, we know~\eqref{eq:Ric*A.renewed},
	\begin{equation*}
		R^*(X,Y)Z = A(Y,Z)X-A(X,Z)Y\,,
	\end{equation*}
	and hence $\Ric^*=(n-1)A$. Equation~\eqref{eq:remains.to.prove.this} follows.
	Because of $n\geq3$, a classical theorem~\cite{nomizu-sasaki,Eisenhart} now ensures that the projective flatness of $\nabla^*$ implies
	\[
		\nabla^*_Y\Ric^*(X,Z)-\nabla^*_X\Ric^*(Y,Z) = 0\,,
	\]
	from which the desired statement is inferred by a direct computation, replacing $\Ric^*$ by $A$, and $\nabla^*$ by $\nabla$.
\end{proof}

\begin{rmk}\label{rmk:proj.flatness}
	In the proof we have drawn upon the deeper fact that Theorem~\ref{thm:existence.uniqueness} can be reformulated in terms of the projective flatness of $\nabla^*$. For details on this `geometric version' of the fundamental theorem we refer the reader to~\cite{DNV1990}, \cite[Ch.\ II, Thm 8.2]{nomizu-sasaki}. 
\end{rmk}

The integrability conditions for the system~\eqref{eq:structure.relative} are thus satisfied and, by the existence part of Theorem~\ref{thm:existence.uniqueness}, it is therefore possible to integrate for $f$ and $\Xi$ (respectively for~$\xi$).
The uniqueness follows from the uniqueness part of Theorem~\ref{thm:existence.uniqueness}.

It only remains to verify that the hypersurface indeed satisfies the axioms~\eqref{eq:RACS.conditions} of an abundant hypersurface. Indeed, \eqref{eq:RACS.conditions.Du} is equivalent to~\eqref{eq:ACSIS.conds.DDt} under the ansatz. Similarly, it is verified by a direct computation that \eqref{eq:RACS.conditions.DU} follows from~\eqref{eq:ACSIS.conds.DS}. The computation is straightforward, but technical, and therefore omitted here. It can be found in Appendix~\ref{app:DU}.
This concludes the proof of Theorem~\ref{thm:main.1}.

\subsection{Proof of Theorem~\ref{thm:main.2}}\label{sec:proof.racs2acsis}

We shall now prove the converse of Theorem~\ref{thm:main.1}: Given an abundant hypersurface $f:M\to\R^{n+1}$ with co-normal $\Xi$, we construct an abundant manifold.
Our strategy is as follows: according to Equations~\eqref{eq:structure.relative}, the hypersurface co-normalisation has the associated connection $\nabla$, the metric $G$ and the Weingarten operator~$A$.
We denote the Levi-Civita connection of $g$ by $\nabla^g$ and proceed in two steps: we first draft an ansatz for the abundant structure. Then, we verify that this ansatz indeed satisfies the axioms of Definition~\ref{defn:ACSIS}.
We consider the hypersurface cubic $C$ and decompose it, as in~\eqref{eq:C2Uu}, into its trace-free part $U$ and trace part $u$,
\[
	C(X,Y,Z) = U(X,Y,Z)+u(X)G(Y,Z)+u(Y)G(X,Z)+u(Z)G(X,Y)\,,
\]
where $u=df$ for some function $f$ due to Lemma~\ref{lem:exactness.T} and since $M$ is simply connected.
We now make the ansatz
\begin{equation}\label{eq:ansatz.RACS}
	S(X,Y,Z):=3\,U(X,Y,Z)\,,\qquad t:=3f\,,\quad g:=G\,,
\end{equation}
where $g$ is understood to be the metric on $M$.

It remains to check that, given the hypersurface integrability conditions~\eqref{eq:integrability} and the conditions~\eqref{eq:RACS.conditions} of an abundant hypersurface (co-)normalisation, the conditions~\eqref{eq:ACSIS.conds} of an abundant structure are fulfilled for the ansatz $(S,t)$ on $(M,g)$.

We now show that with~\eqref{eq:ansatz.RACS}, the conditions~\eqref{eq:ACSIS.conds} are satisfied under the hypothesis of Theorem~\ref{thm:main.2}.

Equation~\eqref{eq:RACS.conditions.Du} immediately implies~\eqref{eq:ACSIS.conds.DDt} using~\eqref{eq:ansatz.RACS}. 	
In order to complete the proof, we have to show that also~\eqref{eq:ACSIS.conds.DS} and~\eqref{eq:ACSIS.conds.Weyl} hold.
Given~\eqref{eq:RACS.conditions.DU} and~\eqref{eq:ansatz.RACS}, \eqref{eq:ACSIS.conds.DS} becomes equivalent to
\begin{equation}\label{eq:conditions.reconstruction}
	\begin{aligned}
			\nabla^g_W S(X,Y,Z) - \nabla^g_Z S(X,Y,W)
			&= \frac13\Big( g(X,Z)\mathring Q(Y,W)+g(Y,Z)\mathring Q(X,W)
			\\
			&\qquad  -g(X,W)\mathring Q(Y,Z) -g(Y,W)\mathring Q(X,Z) \Big)\,,
	\end{aligned}
\end{equation}
noting that $\nabla^gS$ has only one independent trace, where
\begin{equation*}
	Q(X,Y) = \frac2{n-2}\,\mathscr{S}(X,Y)-S(X,Y,\grad_gt)\,.
\end{equation*}
We are hence left with showing that~\eqref{eq:conditions.reconstruction} holds.
Using~\eqref{eq:A.identity.Codazzi}, this immediately reduces to proving
\[
	\div_g(S) = \frac{n}{3}\mathring Q\,.
\]
To show that this condition holds, we compute its left hand side using~\eqref{eq:ansatz.RACS}, Lemma~\ref{lem:useful.C2Uu.identities} and~\eqref{eq:A0},
\[
	\div_g(S) = 3\,\div_G(U) = 3\,\left[ n\nabla^G u -\div_G(u)G -\frac{n}{2}\mathring A \right]\,.
\]
Utilising \eqref{eq:trA} and~\eqref{eq:A.identity.Ric} to eliminate the Weingarten operator $A$, we then find
\[
	\div(S) = 3n\,\left[ \frac{2}{n-2}\,\mathring{\mathscr{U}}-U(\hat u) \right]\,,
\]
where we have also used~\eqref{eq:RACS.conditions.Du}, Lemma~\ref{lem:useful.C2Uu.identities}, and the identity
\[
	C(\hat u)=U(\hat u)+2\,u\otimes u+|u|^2\,g\,.
\]
On the right hand side, we find, using~\eqref{eq:ansatz.RACS},
\begin{align*}
	\frac{n}{3}\mathring{Q} = 3n\left[ \frac{2}{n-2}\mathfrak{\mathring{\mathscr{U}}}-U(\hat u)\right]\,.
\end{align*}
This confirms that~\eqref{eq:conditions.reconstruction} holds, and therefore~\eqref{eq:ACSIS.conds.DS} is satisfied (more details are given in Appendix~\ref{app:DS}).
To complete the proof, we note that the conditions of an abundant hypersurface include~\eqref{eq:A.identity.Weyl}, which immediately implies~\eqref{eq:ACSIS.conds.Weyl.alternative}.
Together with~\eqref{eq:ACSIS.conds.DS}, this implies~\eqref{eq:ACSIS.conds.Weyl}, see Remark~\ref{rmk:Weyl.projector}.
Therefore Theorem~\ref{thm:main.2} is proven.

\section{Conformal equivalence}\label{sec:conformal}

\subsection{Conformal transformations of abundant manifolds}\label{sec:conformal.transformations.ACSIS}

We define conformal rescalings of abundant manifolds as follows.
\begin{defn}\label{defn:ACSIS.conf.equivalence}
	Two abundant manifolds $(M,g,S,t)$ and $(M,g',S',t')$ are said to be \emph{conformally equivalent} if
	\begin{align*}
			g' &= \Omega^2 g\,,
			&
			\hat S' &= \hat S\,,
			&
			t' &= t-3\ln|\Omega|
	\end{align*}
	for some nowhere vanishing function $\Omega\in\mathcal C^\infty(M)$.
\end{defn}

Note that this definition of conformal equivalence is weaker than the equivalence introduced in Definition~\ref{defn:ACSIS.equivalence}.

\begin{rmk}
	In addition to Remark~\ref{rmk:ACSIS}, we note that this definition is aligned with the definition of conformal rescalings of second order (maximally) conformally superintegrable systems \cite{KSV2024}. In the case of second order properly superintegrable systems, these are closely related to the classical St\"ackel transform \cite{KKM-2,KKM-4}, which can be traced back to the classical Maupertuis principle \cite{Maupertuis_1750,jacobi,Tsiganov2001}, and is related to coupling constant metamorphosis \cite{HGDR1984}. Classical St\"ackel transforms and coupling constant metamorphosis are equivalent in the context of second order superintegrability, but the concepts differ in general \cite{Post10}.
	
	Note that Definition~\ref{defn:ACSIS.conf.equivalence} is consistent with~\cite{KSV2024}: if the metric undergoes a rescaling $g\longto\Omega^2g$, then the superintegrable structure tensor $S$ transforms according to $\hat S\longto\hat S$ and $S\longto\Omega^2S$, while $t$ follows the transformation rule $t\longto t-3\ln|\Omega|$.	
\end{rmk}

We now summarise some transformation rules associated with conformal rescalings of abundant manifolds.
From standard conformal geometry, cf.~\cite{Kulkarni69} for instance, we know that the Levi-Civita connections of $g'$ and $g$ are related by
\begin{align*}
	\nabla^{g'}_XY &= \nabla^g_XY +\Upsilon(X)Y+\Upsilon(Y)X-g(X,Y)\,\hat\Upsilon\, , \quad \Upsilon = d\ln |\Omega |\, .
\end{align*}
We therefore find that the cubics of the associated abundant hypersurface normalisations are related according to, compare \eqref{eq:ansatz.C},
\begin{align*}
	\hat C'(X,Y)
	&=\tfrac13\left(\hat S'(X,Y)+X(t')Y+Y(t')X+g'(X,Y)\grad_{g'}t'\right) \\
	&=\hat C(X,Y) - \left(\Upsilon(X)Y+\Upsilon(Y)X+g(X,Y)\hat\Upsilon\right)\,.
\end{align*}
Since $\hat A'$ and $\hat A$ are implied by $G',C'$ and $G,C$, respectively, via~\eqref{eq:Ric.ast.A}, we will now immediately proceed to a `dual' viewpoint and consider, in the next subsection, conformally equivalent relative affine hypersurfaces, following the established concept of conformal rescalings in affine hypersurface geometry.
In Section~\ref{sec:conformal.ACSIS.RACS}, we shall then compare the two conformal rescalings, observing their compatibility.

\subsection{Conformal transformations of affine hypersurfaces}\label{sec:conformal.transformations.hypersurfaces}

In this section we consider the relative affine hypersurface equations~\eqref{eq:structure.relative} and their integrability conditions~\eqref{eq:integrability} from the viewpoint of conformal rescalings on the manifold $M$. We introduce the conformal equivalence of relative affine hypersurfaces following the definitions in \cite{SSV1991}.
Let $(f:M\to\R^{n+1}\,,\,\Xi)$ be a relative affine hypersurface co-normalisation, and consider another co-normalisation $(f,\Xi')$ with $\Xi'\ne\Xi$.
Since $C(M)$ is a 1-dimensional linear space, there must be a function $\sigma$ such that $\Xi'=\sigma\Xi$.
The following definition ensures that the signature of the metric $G$ is preserved.
\begin{defn}\label{defn:conf.equiv.normalisations}
	Let $f:M\to\R^{n+1}$ be a relative non-degenerate affine hypersurface.
	
	(i) Two co-normalisations $(f,\Xi)$ and $(f,\Xi')$ are said to be \emph{conformally equivalent} if there is a function $\Omega:M\to\R$ on $M$ with $\Omega\ne0$ and such that
	\[ \Xi'=\Omega^2\Xi\,. \]
	
	(ii) Two normalisations $(f,\xi)$ and $(f,\xi')$ are said to be \emph{conformally equivalent} if their associated co-normalisations $(f,\Xi)$ and, respectively, $(f,\Xi')$ are conformally equivalent.	
\end{defn}
It follows immediately from \eqref{eq:barnabla} that the associated metrics $G$ and $G'$ of the relative affine hypersurface (co-)normalisations are related by $G'=\Omega^2G$. Next we compare the underlying normalisations $\xi$ and $\xi'$ (and later 
the other data, compare \cite{SSV1991}). 
To this end, let $(f:M\to\R^{n+1},\Xi)$ and $(f':M\to\R^{n+1},\Xi')$ be two conformally equivalent affine hypersurface co-normalisations.
The structural equations are of the form
\begin{equation}\label{eqs:comparison.conformal.rescalings.hypersurfaces}
\begin{aligned}
	\bar\nabla_XY &= \nabla_XY+G(X,Y)\xi\,,
	&
	\qquad\qquad
	\bar\nabla_X\xi &= -\hat A(X)\,,
	\\
	\bar\nabla_XY &= \nabla'_XY+\Omega^2 G(X,Y)\xi'\,,
	&
	\bar\nabla_X\xi' &= -\hat A'(X)\,,
\end{aligned}
\end{equation}
where $\xi$ and $\xi'$ are the transversal fields associated to $\Xi$ and $\Xi'$, respectively, via Corollary~\eqref{cor:xi<->Xi}. The associated connections are $\nabla,\nabla'$, respectively.
Since $\bar\nabla_XY$ does not depend on the choices of $\Xi$ and $\Xi'$, we get, in particular,
\[
	\nabla_XY+G(X,Y)\xi = \bar\nabla_XY = \nabla'_XY+\Omega^2 G(X,Y)\xi'\,.
\]
The transversal fields need to satisfy~\eqref{eq:co-normalisation.horizontal},
\begin{equation}\label{eq:normalisation.Xi.xi}
	\langle\xi',\Xi'\rangle=1=\langle\xi,\Xi\rangle\,,
\end{equation}
and thus, using Definition~\ref{defn:conf.equiv.normalisations}, we conclude
\begin{equation}\label{eq:xi.pre-transform}
	\xi'=\Omega^{-2}\xi+w\,,
\end{equation}
where $w\in\Gamma(TM)$ is determined by 
\begin{equation}\label{eq:condition.w.connections}
	\Omega^2G(X,Y)w = \nabla_XY-\nabla'_XY\,.
\end{equation}
The vector field~$w$ can also be specified in the following form, see \cite[\S~5.1]{SSV1991}, cf.\ \cite[Ch.\ II, Prop.\ 2.5]{nomizu-sasaki}. 

\begin{lem}
\label{lem:conf_co-norm}
	If $\Xi'=\Omega^2\Xi$ are two relative co-normalisations, then their corresponding (canonical) transversal fields satisfy
	\[ \xi'=\Omega^{-2}\left( \xi+2\,\grad_G\ln|\Omega| \right)\,. \]
\end{lem}

We now return to the condition~\eqref{eq:condition.w.connections}. Rewritten in terms of the Levi-Civita connections of $G'$ and $G$, respectively, it becomes
\[
	\nabla^{G'}_XY + \hat C'(X,Y) + 2\,G(X,Y)\,\hat\Upsilon =  \nabla^G_XY +\hat C(X,Y)\,.
\]
On the other hand, from the theory of conformally equivalent connections we know
\[
	\nabla^{G'}_XY =  \nabla^G_XY + \Upsilon(X)Y + \Upsilon(Y)X - G(X,Y)\hat\Upsilon\,.
\]
The hypersurface cubic therefore has to transform according to
\begin{align}
	\hat C'(X,Y) &= \hat C(X,Y)-\Upsilon(X)Y-\Upsilon(Y)X-G(X,Y)\,\hat\Upsilon\,,
	\label{eq:C.conf.transf}
	\\
	\intertext{which implies}
	C'(X,Y,Z) &= \Omega^2\left[ C(X,Y,Z)-\Upsilon(X)G(Y,Z)-\Upsilon(Y)G(X,Z)-\Upsilon(Z)G(X,Y) \right]
	\nonumber
	\\
	\intertext{and then}
	\tr\hat C'(X) &= \tr\hat C(X) -(n+2)\Upsilon(X)\,.
	\nonumber
\end{align}
Note that this formula means that the one-form $u$ and the trace-free totally symmetric tensor $U$ transform as $u'=u-\Upsilon$ and 
$U'=\Omega^2 U$. 
Reviewing~\eqref{eqs:comparison.conformal.rescalings.hypersurfaces}, it remains to consider the transformation rule for the Weingarten operator $\hat A$: 
\begin{align*}
	-\hat A'(X)
	= \bar\nabla_X\xi'
	&= \bar\nabla_X(\Omega^{-2}(\xi+2\hat\Upsilon))
	\\
	&= -2\Omega^{-3}X(\Omega)(\xi+2\hat\Upsilon)
	+\Omega^{-2}\left( \bar\nabla_X\xi+2\bar\nabla_X\hat\Upsilon \right)
	\\
	&= \Omega^{-2}\left( -4\Upsilon(X)\hat\Upsilon -\hat A(X) +2\nabla_X\hat\Upsilon \right)
\end{align*}
and thus conclude
\[
	\hat A'(X) = \Omega^{-2}\left( \hat A(X) + 4\Upsilon(X)\hat\Upsilon-2\nabla_X\hat\Upsilon \right)\,.
\]
Using the metric $G'=\Omega^2G$ to convert $\hat A'$ into a bilinear form, we infer
\[ A'(X,Y) = A(X,Y) + 4\Upsilon(X)\Upsilon(Y)-2\,\nabla^2\left(\ln|\Omega|\right)(X,Y)\,. \]

\subsection{Conformal rescalings of abundant hypersurface normalisations and abundant manifolds}\label{sec:conformal.ACSIS.RACS}

Having characterised the correspondence between abundant manifolds and abundant hypersurface normalisations in Section~\ref{sec:ACSIS.RACS}, and with conformal rescalings introduced for either of these structures in Sections~\ref{sec:conformal.transformations.ACSIS} and~\ref{sec:conformal.transformations.hypersurfaces}, we shall now show that these two definitions are compatible.
Specifically, we will prove the following statement about the conformal equivalence of abundant manifolds and abundant hypersurface (co-) normalisations.
\begin{thm}\label{thm:main.3} Assume that $M^n$ is a simply connected, smooth manifold, $n\geq3$.
	\begin{enumerate}[label=(\roman*)]
		\item
		Let $(M,g,S,t)$ and $(M,g',S',t')$ be two conformally equivalent abundant manifolds. Then the associated abundant hypersurface normalisations are conformally equivalent.
		\item 
		Let $(f:M\to\R^{n+1},\Xi)$ and $(f,\Xi')$ be two conformally equivalent abundant hypersurface normalisations.
		Then the associated abundant manifolds are conformally equivalent.
	\end{enumerate}
\end{thm}
\begin{proof}
Let us observe that, due to Theorem~\ref{thm:existence.uniqueness}, it suffices to check the following for the respective cases:
\begin{enumerate}[label=(\roman*)]
	\item Let $C$ and  $C'$ denote, respectively, the cubics obtained for the abundant manifold of the hypothesis.	Then their (unique) associated abundant hypersurface (co-)normalisations with these cubics (and the obvious metrics $G$ and $G'$) are conformally equivalent.
	\item Let $U,u$ and $U',u'$ be the tensors as in~\eqref{eq:C2Uu}, respectively, for $(f,\Xi)$ and $(f,\Xi')$.
	Then let $S$ and $S'$ as well as $t$ and $t'$, respectively, be defined as in the ansatz~\eqref{eq:ansatz.RACS}. Hence, due to Theorem~\ref{thm:main.2}, $(M,g=G,S,t)$ and $(M,g'=G',S',t')$ define the associated abundant manifolds. Then there is a function $\Omega$ such that $S'=\Omega^2 S$ and $dt'=dt-3d\ln|\Omega|$ as well as $g'=\Omega^2g$.
\end{enumerate}
We verify these conditions by direct computations, which will be carried out now.\medskip

\textsl{Proof of part (i):}
We have
\[
	g'=\Omega^2 g\,,\qquad
	\hat S'=\hat S\,,\qquad
	dt' = dt-3\Upsilon\,,
\]
where $\Upsilon=d\ln|\Omega|$.
The associated hypersurface normalisations $(f,\Xi)$ and, respectively, $(f,\Xi')$ then satisfy
$ G'=g'$ and $G=g$, and thus $G'=\Omega^2G $, as well as
\[
	\hat C'(X,Y)=\hat C(X,Y)-\Upsilon(X)Y-\Upsilon(Y)X-g(X,Y)\hat\Upsilon
\]
(see \eqref{eq:ansatz.C}) or, in terms of connections,
\[
	\nabla'_XY = \nabla_XY-\Upsilon(X)Y-\Upsilon(Y)X-g(X,Y)\hat\Upsilon\,,
\]
which are consistent with~\eqref{eq:C.conf.transf}. Because of Theorem~\ref{thm:existence.uniqueness}, we can conclude that 
the hypersurface immersion $f$ associated with the data $G, C$ coincides with the one associated with the data 
$G', C'$, where the co-normalisations are related by 
$\Xi'=\Omega^2\Xi$, i.e.~the associated hypersurfaces are conformally equivalent.
This concludes part (i) of the proof.
\medskip

\textsl{Proof of part (ii):}
Two affine hypersurface normalisations $(f:M\to\R^{n+1},\Xi')$ and $(f:M\to\R^{n+1},\Xi)$ are conformally equivalent if and only if
\[ \Xi'=\Omega^2\Xi \]
for a function $\Omega\in\mathcal C^\infty(M)$. As usual, we denote the data associated to $(f, \Xi)$ by $(G, C)$ and similarly for $(f, \Xi')$. The metrics $g, g'$ of the corresponding abundant manifolds are conformally related: 
\[
	g'=G'=\Omega^2 G,\quad g= G\,.
\]
As shown in the previous section, the tensors $(U,u)$ and $(U',u')$ associated with $C$ and $C'$, respectively, obey the transformation rules
\[
	\hat U'(X,Y)=\hat U(X,Y)\,,\qquad u'(X)=u(X)-\Upsilon(X)\,,
\]
where $X,Y\in\mathfrak X(M)$.
These transformation rules are consistent with those for the data $(S,dt)$, $(S',dt')$ of the the corresponding abundant manifolds, see Definition~\ref{defn:ACSIS.conf.equivalence} and \eqref{eq:ansatz.RACS}. This concludes the proof of part (ii).
\medskip

Theorem~\ref{thm:main.3} is hence proven.
\end{proof}

\begin{rmk}
	The conditions~\eqref{eq:RACS.conditions} are conformally invariant.\smallskip
	
	Indeed, recalling $\hat U\to\hat U$ and $u\to u-\Upsilon$, we obtain, via the standard transformation formulas for conformally equivalent metrics,
	\[
	\nabla^Gu\to\nabla^G(u-\Upsilon)-\Upsilon\otimes(u-\Upsilon)-(u-\Upsilon)\otimes\Upsilon+\langle\Upsilon,u-\Upsilon\rangle\,g\,.
	\]
	Since the Schouten tensor transforms according to
	\[
	\P^G\to \P^G-\nabla\Upsilon+\Upsilon\otimes\Upsilon-\frac12\,|\Upsilon|^2_G\,G\,,
	\]
	it then follows that~\eqref{eq:RACS.conditions.Du} is invariant. The proof that~\eqref{eq:RACS.conditions.DU} is conformally invariant is analogous.
	Note that the invariance of the conditions~\eqref{eq:RACS.conditions} also follows directly from the correspondence of abundant manifolds and abundant hypersurfaces (Theorems~\ref{thm:main.1} and~\ref{thm:main.2}), invoking the conformal invariance of the corresponding conditions for abundant manifolds, which is shown in \cite{KSV2024}.
\end{rmk}

Since the conditions~\eqref{eq:RACS.conditions} are conformally invariant, we have that if these conditions hold for a relative affine hypersurface (co-)normalisation, then they also hold for any conformally equivalent (co-)normalisation, i.e.~the property of being abundant is preserved under conformal rescalings.
Hence Theorem~\ref{thm:main.3} provides us with a concept of \emph{abundant hypersurfaces}, independent of the choice of relative (co-)normalisation:
\begin{defn}\label{defn:RACS.hypersurface}
	We say that an affine hypersurface is an \emph{abundant hypersurface}, if it admits an abundant hypersurface (co-)normalisation.
\end{defn}
Note that for an abundant hypersurface, any relative (co-)normalisation is an abundant hypersurface (co-)normalisation.

\begin{rmk}\label{rmk:equiaffine}
	In \cite[\S 5.3.1.3 and \S 6.2]{SSV1991} it is proven that for an (oriented) affine hypersurface $f:M\longto\R^{n+1}$ there is a unique relative affine hypersurface co-normalisation 	$(f,\Xi)$ that is Blaschke.
	
	Under the correspondence~\eqref{eq:ansatz.RACS}, homothetically Blaschke co-norma\-lisa\-tions correspond to abundant structures with $dt=0$ (and vice versa). This latter case has been discussed in the context of second-order conformally superintegrable systems in \cite{KSV2024,KSV2024_bauhaus}, and is referred to by the terms ``standard gauge'' or ``standard scale'', often assuming $t=0$ due to Definition~\ref{defn:ACSIS.equivalence}.
\end{rmk}

We conclude the section with the following example, which is inspired by \cite{KSV2024} and by~\cite[\S 7.2.2]{SSV1991}.
\begin{prop}~
	\begin{enumerate}[label=(\roman*)]
		\item Consider a (simply connected) abundant manifold with $n\geq3$ that satisfies $u=0$, and such that $\mathring{\P}^g=0$ holds for the Schouten tensor of the metric~$g$.
		Then the corresponding abundant hypersurface co-normalisation is an affine sphere (for an appropriate constant volume form on $\R^{n+1}$).
	
		\item For simply connected $M^n$ with $n\geq3$, consider an abundant hypersurface co-norma\-li\-sa\-tion $f:M\to\R$ which is an affine sphere.
		Then the Schouten tensor of the metric of the corresponding abundant manifold satisfies $\mathring{\P}^g=0$.
	\end{enumerate}
\end{prop}
\begin{proof}
	We can build on the reasoning in Section~\ref{sec:ACSIS.RACS}, which yields, for $u=0$,
	\begin{align*}
		\div_G(U) &= \tfrac{2n}{n-2}\,\mathring{\mathscr{U}}\,,
		&
		\mathring{\P}^G &= -\frac1{n-2}\mathring{\mathscr{U}}\,,
	\end{align*}
	compare~\eqref{eq:div(U).formula} and \eqref{eq:RACS.conditions.Du}.
	The first claim now follows immediately from~\eqref{eq:A0}, implying $\mathring{A}=-\frac2n\div_G(U)=4\mathring{\P}^G$, compare Lemma~\ref{lem:relative.sphere.characterisation}.
	For the second claim, note that $u=0$ since an affine sphere necessarily requires a Blaschke normalisation, cf.\ Section~\ref{sec:equiaffine}. With~\eqref{eq:A.identity.Ric}, we then infer
	\[
		0=\frac{1}{2}\mathring{A}=\mathring{\P}^G-\frac{1}{n-2}\mathring{\mathscr U}\,.
	\]
	The claim then follows due to~\eqref{eq:RACS.conditions.Du}, which yields
	$\mathring{\P}^G+\frac{1}{n-2}\mathring{\mathscr U}=0\,.$
\end{proof}

\section{The correspondence in dimension 2}\label{sec:surfaces}

In the theory of non-degenerate (conformal) superintegrability, it is well known that the structural equations of two-dimensional systems differ significantly from those in higher dimensions, and are in some respects more involved.
In the previous sections we have therefore focused our attention on dimensions $n\geq3$ of the underlying manifold. The case $n=2$ is going to be discussed now. Our departing point is \cite{KSV2024_bauhaus}, where the theory of two-dimensional non-degenerate second-order conformally superintegrable systems is developed, extending \cite{KSV2023,KSV2024} and the earlier works~\cite{KKM-1,KKM-2,Kress&Schoebel,KMK2007}.

\begin{rmk}
	We remark that, in dimension two, all non-degenerate systems are abundant and all superintegrable systems are maximally superintegrable, for dimensional reasons. The results in the present section will therefore apply to any non-degenerate second-order system in dimension two \cite{KKM-1}.
\end{rmk}

We define abundant manifolds in dimension two as follows.

\begin{defn}\label{defn:ACSIS.2D}
	Let $(M,g)$ be a (pseudo-)Riemannian (oriented) manifold of dimension $n=2$.
	Assume it is equipped with a totally symmetric and trace-free $(0,3)$-tensor field $S$ and a smooth function $t$.
	We say that $(M,g,S,t)$ is an \emph{abundant manifold} if, with $X,Y,Z,W\in\X(M)$,
	\begin{subequations}\label{eq:ACSIS.conds.2D}
		\begin{align}
			(\nabla^g_WS)(X,Y,Z)
			&=  \Pi_{\mathrm{Sym^3_0}}\left[ -\frac23\,S\otimes dt+2\,dt\otimes S+\sha\otimes g\right](X,Y,Z,W)\,,
			\label{eq:DS.2D}
			\\
			(\nabla^g_Z\sha)(X,Y)
			&= \frac13\,|S|^2\, S(X,Y,Z)
				+\frac43\,\Pi_{\mathrm{Sym^2}}\left( \beta\otimes g(-,Z)-\frac12\,\beta(Z)\,g\right)(X,Y)
			\nonumber
			\\
			&\qquad\qquad
				+\frac43\,\Pi_\mathrm{Sym^3}\left(\sha\otimes dt-\frac12g\otimes \sha(\grad_gt)\right)(X,Y,Z)
			\label{eq:DXi.2D}
			\\ \intertext{and}
			\div_g(\bigtau)
			&= -\eta + \beta -\frac23\,\sha(\grad_gt) -\frac49\,S(\grad_gt,\grad_gt) -\frac59\,|S|^2\,dt
			\nonumber
			\\
			&\qquad\qquad +\frac12d\Scal^g -\Scal^g\,dt\,,
			\label{eq:divTau.2D}
		\end{align}
		where
		\begin{align}
			\bigtau &:= \frac23(\nabla^g)^2t-\frac23\,S(\grad_gt)+\frac13\,\sha
			-\frac89\,\left( dt\otimes dt-\frac12\,|dt|^2\,g \right)
			-\frac19\,|S|^2\,g-\frac12\,\Scal^g\,g
			\label{eq:DDt.2D}
		\end{align}
		is a trace-free and symmetric tensor field, $\tau\in\mathrm{Sym}^2_0(T^*M)$.
		For a concise notation, we have also introduced the auxiliary tensor field
		\begin{align}
			\sha &:= \div_g(S) +\frac23\,S(\grad_gt)\quad\in\mathrm{Sym}^2_0(T^*M)\,.
				\label{eq:sha}		
		\end{align}
		Moreover, we have defined the $1$-forms
		\begin{align}
			\label{eq:beta} \beta(X) &:= \tr_g( S(X,\cdot,\hat\sha(\cdot)\,)\,)\,,
			\\
			\eta(X) &:=  \tr_g( S(X,\cdot,\hat\tau(\cdot)\,)\,)\,,
		\end{align}
		where $X\in\X(M)$.
	\end{subequations}
\end{defn}
Note that, applying the rules for conformal rescalings of abundant manifolds discussed in Section~\ref{sec:conformal.transformations.ACSIS}, $\sha$ is invariant under such rescalings, but $\bigtau$ is not, cf.~\cite{KSV2024_bauhaus}.

\begin{rmk}
	Solving~\eqref{eq:DDt.2D} for the Hessian of $t$, we obtain
	\begin{align*}
		(\nabla^g)^2t
		&= \frac32\left[
		\bigtau-\frac13\sha+\frac23\,S(\grad_gt)+\frac89\,\left(dt\otimes dt-\frac12\,|dt|^2\,g\right)
		\right]
		+\frac12\,
		\left[
		\frac13\,|S|^2+\frac32\Scal^g
		\right]\,g\,,
	\end{align*}
	where the first pair of square brackets encloses a trace-free expression.
	
	Definition~\ref{defn:ACSIS.2D} is motivated by the structural equations of non-degenerate second-order conformally superintegrable systems in dimension~2 \cite{KSV2024_bauhaus}, which are obtained similarly to those in dimensions $n\geq3$, see Remark~\ref{rmk:ACSIS}. However, note the significant formal difference between Definition~\ref{defn:ACSIS.2D} and Definition~\ref{defn:ACSIS}.
\end{rmk}

\begin{rmk}	
	We tacitly adopt for the case of $2$-dimensional abundant manifolds the concepts of \emph{equivalence} and \emph{conformal equivalence}, analogously to higher dimensions, cf.\ Definitions~\ref{defn:ACSIS.equivalence} and~\ref{defn:ACSIS.conf.equivalence}, respectively.
\end{rmk}

\begin{defn}\label{defn:RACS.2D}
	Let $M\subset\R^3$ be a relative, non-degenerate hypersurface with conormal field $\Xi$. We denote its connection by $\nabla$, its affine metric by $G$ and refer to \eqref{eq:C2Uu} for the decomposition of 
	the associated cubic $C$ in terms of $U$ and $u$.
	
	Then we say that $(f,\Xi)$, respectively its associated hypersurface normalisation $(f,\xi)$, is \emph{abundant} if the (conformally invariant) conditions
	\begin{subequations}\label{eq:2D.RACS.conditions}
		\begin{align}
			\Pi_{\mathrm{Sym}^4_0}(\nabla^GU)
			&= 4\,\Pi_{\mathrm{Sym}^4_0}\,(U\otimes u)\,,
			\label{eq:2D.RACS.condition.DU}
			\\
			\Pi_{\mathrm{Sym}^3_0}\,\left[ \nabla^G\shta-4\,\shta\otimes u \right]
			&= 9\,|U|^2_G\,U\,,
			\label{eq:2D.RACS.condition.DdivU}
			\\
			\div_G(u)-\frac12 \Scal^G
			&=|U|_G^2\,
			\label{eq:2D.RACS.condition.divu}
		\end{align}
	\end{subequations}
	are satisfied in addition to the hypersurface equations~\eqref{eq:structure} and~\eqref{eq:integrability}.
	Here we  use 
	\[ \shta:=3\left( \div_G(U)+2\,U(\hat u) \right)\,. \]
\end{defn}

\begin{rmk}\label{rmk:abundant_hyp}
	The conditions of an abundant hypersurface co-normalisation are conformally invariant.
	Indeed, consider a conformal rescaling with $\Xi\to\Omega^2\Xi$, i.e.\ such that $G\to G'=\Omega^2G$.
	Then, note that, for any $1$-form $\omega$,
	\begin{equation}\label{eq:conformal.transformation.oneform}
		\nabla^G_Y\omega(X) \to
		\nabla^{G'}_Y\omega(X) = \nabla^{G}_Y\omega(X) - \Upsilon(X)\omega(Y) - \Upsilon(Y)\omega(X) + \omega(\hat\Upsilon)G(X,Y)\,,
	\end{equation}
	where $\Upsilon=d\ln|\Omega|$.
	We hence obtain
	\begin{align*}
		\div_{G'}(\Omega^2U)
		&= 2\,U(\hat\Upsilon)+\div_G(U)-4\,U(\hat\Upsilon)+(n+2)\,U(\hat\Upsilon)
		= \div_G(U)+2\,U(\hat\Upsilon)\,
	\end{align*}
	and conclude, since $u'=u-\Upsilon$, that $\shta$ is invariant.
	Similarly, the conformal invariance of~\eqref{eq:2D.RACS.condition.DU} is verified immediately using~\eqref{eq:conformal.transformation.oneform}.
	It remains to check the invariance of~\eqref{eq:2D.RACS.condition.DdivU} and~\eqref{eq:2D.RACS.condition.divu}.
	We have $U\to U'=\Omega^2U$ and, invoking~\eqref{eq:conformal.transformation.oneform},
	\begin{align*}
		\nabla^{G}_Z\shta(X,Y)&\to \nabla^{G'}_Z\shta(X,Y)
		=\nabla^{G}_Z\shta(X,Y)-2\,\Upsilon(Z)\shta(X,Y)
		\\ &\qquad\qquad\qquad\qquad
		-\Upsilon(X)\shta(Z,Y)-\Upsilon(Y)\shta(X,Z)
		\\ &\qquad\qquad\qquad\qquad
		+\shta(\hat\Upsilon,Y)G(X,Z)+\shta(X,\hat\Upsilon)G(Z,Y)\,,
	\end{align*}
	where $X,Y,Z\in\mathfrak X(M)$, as well as
	\begin{align*}
		\shta \otimes u &\to \shta \otimes ( u-\Upsilon )\,,
		\\
		\div_G(u) &\to \Omega^{-2}\left( \div_G(u) -\div_G(\Upsilon) \right)\,,
		\\
		\J^G &\to \Omega^{-2}\left( \J^G-\div_G(\Upsilon) \right)\,,
		\\
		|U|_G^2 &\to \Omega^{-2}|U|_G^2\,,
	\end{align*}
	and hence we conclude that~\eqref{eq:2D.RACS.condition.DdivU} and~\eqref{eq:2D.RACS.condition.divu} are indeed conformally invariant.
	We remark that the conformal invariance of the conditions~\eqref{eq:ACSIS.conds.2D} of an abundant manifold has already been obtained in~\cite{KSV2024_bauhaus}.
\end{rmk}

Our aim is to prove a correspondence similar to that established by Theorems~\ref{thm:main.1} and~\ref{thm:main.2}.
To this end, for a given abundant manifold we make the following ansatz for the cubic $C$, the metric $G$ and the Weingarten form $A$ associated to the hypersurface:
\begin{subequations}\label{eq:ansatz.2D}
	\begin{equation}
		G=g\,,
		\label{eq:ansatz.2D.0}
	\end{equation}
	\begin{equation}
		C(X,Y,Z) = \frac13 \left[ S(X,Y,Z) + g(X,Y)Z(t)+g(Y,Z)X(t)+g(X,Z)Y(t) \right]
		\label{eq:ansatz.2D.1}
	\end{equation}
	and
	\begin{equation}
		A = \frac13\left[
				2(\nabla^g)^2t -\Delta^gt\,g -\div_g(S)
			\right]
		+\frac12\,\left(
			\Scal^g
			-\frac19\,|S|^2
			+\frac{4}{9}\,|dt|^2
		\right)\,g\,.
		\label{eq:ansatz.2D.2}
	\end{equation}
\end{subequations}

Comparing~\eqref{eq:ansatz.2D.1} with~\eqref{eq:C2Uu}, we conclude
\begin{equation*}
	U = \frac13\,S\qquad\text{and}\qquad u = \frac13\,dt\,,
\end{equation*}
which are analogous to the ansatz for dimensions $n\geq3$, cf.~\eqref{eq:ansatz.C} and~\eqref{eq:ansatz.RACS}.

\begin{lem}\label{lem:ansatz.2D.consistent}
	Given~\eqref{eq:ansatz.2D.1}, Equation~\eqref{eq:ansatz.2D.2} is consistent with~\eqref{eq:Ric.ast.A}.
\end{lem}
\begin{proof}
	Note that in dimension two
	\[
		\Ric^G=\frac12\Scal^G\,G\,.
	\]
	Furthermore, combining~\eqref{eq:A0} and~\eqref{eq:trA},
	\begin{align*}
		A &= \frac2n\left[
		(n+2)\nabla^G u - \div_G(C)
		\right]
		+\frac{1}{n(n-1)}\left( \Scal^G - \lvert C\rvert^2 + (n+2)^2\,\lvert u\rvert^2 \right)\,G
		\\
		&= 4\nabla^G u - \div_G(C)
		+\frac{1}{2}\left( \Scal^G - \lvert C\rvert^2 + 16\,\lvert u\rvert^2 \right)\,G
	\end{align*}
	and we observe that~\eqref{eq:A.identity.Ric} is equivalent to~\eqref{eq:trA} for $n=2$. This can be checked using 
	Lemma \ref{lem:useful.C2Uu.identities}. 
	Next, we have
	\begin{align*}
		\div_G(C) &= \div_G(U)+2\,\nabla^Gu + \div_G(u)\,G\,,
	\end{align*}
	and $|C|^2=|U|^2+12\,|u|^2$, and thus infer
	\begin{equation}\label{eq:2D.A}
		A = 2\nabla^G u - \div_G(U) - \div_G(u)G
		+\frac{1}{2}\left( \Scal^G - |U|^2 + 4\,|u|^2 \right)\,G\,.
	\end{equation}
	Using the ansatz~\eqref{eq:ansatz.2D.1}, we arrive at \eqref{eq:ansatz.2D.2} proving the claim.
\end{proof}

\subsection{The correspondence in dimension two}
We now prove the following theorem, which establishes the two-dimensional analog of the correspondence provided by Theorems~\ref{thm:main.1} and~\ref{thm:main.2}.
\begin{thm}\label{thm:2D}~
	\begin{enumerate}[label=(\roman*)]
		\item Let $(M,g,S,t)$ be a 2-dimensional simply connected abundant manifold. Then there is a (unique up to affine transformations) abundant hypersurface co-normalisation $(f,\Xi)$ with immersion $f:M\to\R^3$ and relative co-normalisation $\Xi:M\to\R^3$, whose associated metric and cubic are given by~\eqref{eq:ansatz.2D}.
		\item Let $M$ be simply connected and assume that $(f:M\to\R^3,\Xi)$ is an abundant hypersurface co-normalisation with associated connection $\nabla$ and metric $G$. Let
		\[
			C(X,Y,Z) = \frac12(\nabla_XG)(Y,Z)\,,
		\]
		which is totally symmetric, and decompose $C$ as in~\eqref{eq:C2Uu}. Then there is a function~$t$ on~$M$ such that $(M,G,3U,t)$ is an abundant manifold with $3u=dt$.
	\end{enumerate}
\end{thm}

We prove Theorem~\ref{thm:2D} in the remainder of the section, proceeding similarly as for Theorems~\ref{thm:main.1} and~\ref{thm:main.2}, but with a somewhat modified strategy. In Sections~\ref{sec:proof.acsis2racs} and~\ref{sec:proof.racs2acsis} we verified the relevant axioms explicitly in each case.
Now, we shall prove a restricted version of Theorem~\ref{thm:2D} first. We shall then utilise conformal rescalings of abundant manifolds and abundant hypersurface normalisations, compare Section~\ref{sec:conformal}, to extend the restricted version to the desired statement.

Our first step is to prove Theorem~\ref{thm:2D} for the case of \emph{standard-scale} systems, i.e.\ $dt=0$, on the one hand, and for \emph{homothetically Blaschke}  normalisations, i.e.~$u=0$, on the other:
\pagebreak[3]
\begin{lem}\label{lem:2D.standard.equiaffine}~
	\begin{enumerate}[label=(\roman*)]
		\item
		Let $(M,g,S,t)$ be a 2-dimensional abundant manifold in standard scale, i.e.\ $dt=0$. Then there is an immersion $f:M\to\R^3$ with relative co-normalisation $\Xi$ (and associated canonical normalisation $\xi$) such that $(f,\xi)$ is an abundant homothetically Blaschke co-normalisation, whose associated metric and cubic are given by~\eqref{eq:ansatz.2D}.
		\item 
		Let $M$ be simply connected and assume that $(f:M\to\R^3,\Xi)$ is an abundant homothetically Blaschke co-normalisation with associated connection $\nabla$ and metric~$G$. Let
		\[
		C(X,Y,Z) = \frac12\nabla_XG(Y,Z)\,,
		\]
		which is totally symmetric, and decompose $C$ as in~\eqref{eq:C2Uu}. Then $(M,G,3U,0)$ is an abundant manifold in standard scale.
	\end{enumerate}
\end{lem}

\begin{proof}[Proof of Lemma~\ref{lem:2D.standard.equiaffine}]
	\textsc{Proof of the direction ``abundant manifold $\Rightarrow$ abundant hypersurface (co-)normalisation'':}
	We need to verify that the integrability conditions~\eqref{eq:integrability} are satisfied for the ansatz~\eqref{eq:ansatz.2D} under the hypothesis of the claim.	
	It is obvious that~\eqref{eq:integrability.C} and~\eqref{eq:integrability.dTheta} hold.
	Now consider~\eqref{eq:integrability.gauss.1} and~\eqref{eq:integrability.gauss.2}.
	Alternatively, we may consider~\eqref{eq:A0}--\eqref{eq:A.identity.Codazzi}. Equations~\eqref{eq:A0}--\eqref{eq:A.identity.Ric} are automatically satisfied by~\eqref{eq:ansatz.2D}. Indeed Equations~\eqref{eq:A0} and \eqref{eq:trA} follow directly by substitution of \eqref{eq:2D.A}	with $dt=0$ and for Equation~\eqref{eq:A.identity.Ric} one uses, in addition,  that the trace-free Ricci tensor vanishes for $2$-dimensional
	Riemannian manifolds together with the identity $\mathscr C = \frac12 |C|_g^2g$, which holds for any trace-free cubic $C$. 
	In dimension~2, \eqref{eq:A.identity.Weyl} is trivially satisfied (since the Weyl tensor vanishes in dimension 2). So we are left with~\eqref{eq:A.identity.Codazzi}. It is satisfied by the ansatz~\eqref{eq:ansatz.2D} due to~\eqref{eq:DS.2D} and $dt=0$.
	Finally, we need to verify~\eqref{eq:integrability.codazzi}. This condition rewrites as
	\[
		(\nabla^G_X A)(Y,Z)-(\nabla^G_Y A)(X,Z) = U(\hat{A}(X),Y,Z)-U(\hat{A}(Y),X,Z), 
	\]
	since $C=U$ due to $3u=dt=0$.
	We observe that, for dimensional reasons, the contracted condition
	\[
		(\div_G \mathring A)(Y)-\frac12\nabla^G_Y\tr \hat{A} - \alpha(Y) = 0\,.
	\]
	is equivalent to it, where $\alpha(Y) := \tr(U_Y\hat A)$.
	We shall now verify that it is satisfied under our hypotheses, which will imply the existence of an affine hypersurface normalisation.
	Indeed, we have, due to \eqref{eq:2D.A} and $u=0$,
	\begin{equation*}
		A=\frac12\Scal^G\,G-\frac{1}{2}\,|U|_G^2\,G-\frac13\,\shta
		=\frac12\Scal^g\,g-\frac{1}{18}\,|S|_g^2\,g -\frac13\,\sha\,.
	\end{equation*}
	Moreover, from~\eqref{eq:DDt.2D} we infer
	\begin{equation}\label{eq:aux.std}
		\tau=\frac13\sha=\frac13\shta\qquad\text{and}\qquad
		\Scal^g=-\frac29\,|S|^2_g\,,
	\end{equation}
	and conclude that
	\[
		A = -\tau+\frac34\,\Scal^g\,g\,,
	\]
	which in turn implies $\beta = -9 \alpha$ due to~\eqref{eq:ansatz.2D.1}.
	We obtain
	\begin{align*}
		\div_G(\mathring A)-\alpha-\frac12\nabla^G\tr\hat A
		&=-\div_g(\tau)+\frac19\beta-\frac34\,d\Scal^g\,.
	\end{align*}
	We also have, invoking~\eqref{eq:DXi.2D} and~\eqref{eq:aux.std},
	\[
		\div_g(\tau)=\frac13\,\div_g(\sha)=\frac49\,\beta\,,
	\]
	and hence
	\begin{align*}
		\div_G(\mathring A)-\alpha-\frac12\nabla^G\tr\hat A
		&=-\frac49\,\beta+\frac19\beta-\frac34\,d\Scal^g\,.
	\end{align*}
	It remains to compute
	\begin{align*}
		d\Scal^g &= -\frac29\,d|S|_g^2 = -\frac49\,\beta\,,
	\end{align*}
	where we have used~\eqref{eq:aux.std} and, in the last step, \eqref{eq:DS.2D}.
	Altogether, we have found
	\begin{align*}
		\div_G(\mathring A)-\alpha-\frac12\nabla^G\tr\hat A
		&= -\div_g(\tau)+\frac19\beta-\frac34\,d\Scal^g \\
		&= -\frac49\,\beta+\frac19\beta+\frac13\,\beta
		= 0\,.
	\end{align*}
	The validity of the abundantness conditions~\eqref{eq:2D.RACS.conditions} follows immediately from~\eqref{eq:DS.2D} and~\eqref{eq:DXi.2D}, using that $\tau\in\Sym^2_0(T^*M)$ and that
	\[
		\sha=\div_g(S)+\frac23\,S(\grad_gt)=3\div_G(U)=\shta\,.
	\]

\textsc{Proof of the direction ``abundant hypersurface (co-)normalisation $\Rightarrow$ abundant manifold'':}
	We need to verify that the conditions~\eqref{eq:ACSIS.conds.2D} are satisfied under the hypothesis.
	Note that, due to the hypothesis we have $u=0$, and hence, by virtue of~\eqref{eq:ansatz.2D}, $dt=0$.
	Using these identities, we obtain the conditions~\eqref{eq:ACSIS.conds.2D} of an abundant manifold in the form
	\begin{subequations}\label{eq:RACS2ACSIS.std}
	\begin{align}
		\Pi_{\mathrm{Sym}^4_0}\nabla^gS
		&=  0 \,,
		\label{eq:RACS2ACSIS.std.1}
		\\
		\label{eq:COD_0}
		\Pi_{\mathrm{Codazzi}_0} \nabla^G S &= 0\,,
		\\
		\div_g(S) 
		&=  \sha\,,
		\label{eq:RACS2ACSIS.std.2}
		\\
		\nabla^g\sha
		&= \frac13\,|S|_g^2\,S
		+\frac43\,\Pi_\mathrm{Sym^2_0} \beta\otimes g\,,
		\label{eq:RACS2ACSIS.std.3}
		\\
		\Scal^g &=-\frac29\,|S|^2_g\,,
		\label{eq:RACS2ACSIS.std.4}
		\\
		\tau&=\frac13\sha\,,
		\label{eq:RACS2ACSIS.std.5}
	\end{align}
	\end{subequations}
	where we introduce
	\[
		2\left( \Pi_\mathrm{Sym^2_0} \beta\otimes g\right)(X,Y,Z) :=
		\beta(X)g(Y,Z) + \beta(Y)g(X,Z)-\beta(Z)g(X,Y)\,.
	\]
	We observe that Equation \eqref{eq:COD_0} is automatically satisfied for dimensional reason. In fact 
	$0=\Pi_{\mathrm{Codazzi}_0} \in \mathrm{End}(T^*M\otimes \mathrm{Sym}^3_0T^*M)$. 
	Note that~\eqref{eq:RACS2ACSIS.std.4} and~\eqref{eq:RACS2ACSIS.std.5} correspond to the trace-free and the trace component of~\eqref{eq:DDt.2D}, while~\eqref{eq:divTau.2D} is redundant%
	\footnote{We report two erroneous signs in \cite{KSV2024_bauhaus}. In~\eqref{eq:DXi.2D} and~\eqref{eq:DDt.2D}, the signs of the second and, respectively, the third term on the right hand side have been corrected compared to Proposition~3.2 of the reference.}
	and automatically satisfied after substituting~\eqref{eq:RACS2ACSIS.std.5} and~\eqref{eq:RACS2ACSIS.std.2}, as well as $d\Scal^G=-\frac49\,\beta$, which is obtained similarly to the corresponding reasoning in the first part of the proof.
	Eliminating the auxiliary tensor fields $\sha$ and $\tau$ in the conditions~\eqref{eq:RACS2ACSIS.std}, we obtain the equivalent conditions
	\begin{subequations}\label{eq:R2A.std}
		\begin{align}
			\Pi_{\mathrm{Sym}^4_0}\nabla^gS
			&=  0\,, 
			\label{eq:R2A.std.1}
			\\
			\nabla^g\div_g(S)
			&= \frac13\,|S|_g^2\,S
			+\frac43\,\Pi_\mathrm{Sym^2_0} \beta\otimes g\,,
			\label{eq:R2A.std.3}
			\\
			\Scal^g &= -\frac29\,|S|^2_g\,.
			\label{eq:R2A.std.4}
		\end{align}
	\end{subequations}
	Indeed, substitute~\eqref{eq:RACS2ACSIS.std.2} into~\eqref{eq:RACS2ACSIS.std.3} to eliminate $\sha$.
	The claim follows noting that~\eqref{eq:RACS2ACSIS.std.2} simply defines $\sha$, i.e.\ it does not constitute any additional constraint. Similarly, \eqref{eq:RACS2ACSIS.std.5} defines $\tau$ and does not impose a further condition.\medskip

	We now have to confirm that~\eqref{eq:R2A.std} hold under the hypothesis.
	Due to~\eqref{eq:ansatz.2D} and $u=0$, we have $dt=0$ and $S=3\,U$.
	Furthermore, due to~\eqref{eq:integrability} and Lemma~\ref{lem:gauss.conditions}, we have the equations
	\begin{subequations}\label{eq:integrability.2D}
		\begin{align}
			\Scal^G -|U|^2_G &= \tr(\hat A)\,,
			\label{eq:integrability.2D.gauss.1}
			\\
			-\div_G(U) &= \mathring{A}\,,
			\label{eq:integrability.2D.gauss.2b}
			\\
			\div_G(\mathring{A})-\frac12\,\nabla^G\tr\hat A &= \alpha\,.
			\label{eq:integrability.2D.codazzi}
		\end{align}
	\end{subequations}
	We have omitted the equation \eqref{eq:A.identity.Codazzi}, since it it is automatic in dimension $2$. 
	Furthermore, we have also used the fact that, in dimension~$2$, \eqref{eq:integrability.codazzi} holds precisely if its trace holds, cf.~the first part of this proof.
	Moreover, due to the assumption $u=0$, the abundantness conditions~\eqref{eq:2D.RACS.conditions} simplify to
	\begin{subequations}
		\begin{align}
			\Pi_{\mathrm{Sym}^4_0}(\nabla^GU) &= 0\,,
			\label{eq:2D.RACS.condition.DU.equiaffine}
			\\
			\Pi_{\mathrm{Sym}^3_0}\,\nabla^G\div_G(U)
			&= 3|U|^2_G\,U\,,
			\label{eq:2D.RACS.condition.DdivU.equiaffine}
			\\
			|U|_G^2+\frac12 \Scal^G &= 0\,,
			\label{eq:2D.RACS.condition.divu.equiaffine}
		\end{align}
	\end{subequations}
	which hold in addition to~\eqref{eq:integrability.2D}.
	
	We begin the proof by noting that~\eqref{eq:R2A.std.1} holds true because of~\eqref{eq:2D.RACS.condition.DU.equiaffine}, and~\eqref{eq:R2A.std.4} holds because of~\eqref{eq:2D.RACS.condition.divu.equiaffine}. (Recall that $U=\frac13S$.) 
	Next, consider~\eqref{eq:R2A.std.3}. We first note that, cf.~\eqref{eq:integrability.2D.gauss.1} and \eqref{eq:2D.RACS.condition.divu.equiaffine},
	\[ \tr(\hat A)=-3\,|U|^2_G.\] 
	From~\eqref{eq:integrability.2D.gauss.2b} and \eqref{eq:integrability.2D.codazzi}, we therefore obtain
	\[
		\div_G \div_G(U) =\frac32\,d|U|^2_G -\alpha
		=\frac13\beta+\frac19\beta=\frac49\beta\,,
	\]
	which proves the trace of~\eqref{eq:R2A.std.3}.
	Here we have used that
	\begin{align*}
		\alpha &= -\frac19\beta &\text{and}&&
		d|U|^2_G &= \frac29\beta\,.
	\end{align*}
	It remains to prove that
	\begin{equation}\label{eq:aux1}
		\Pi_{\mathrm{Sym}^3_0}\,\nabla^G\div_G(U) = 3\,|U|^2_G\,U\,,
	\end{equation}
	compare~\eqref{eq:R2A.std.3}, and indeed, this follows from~\eqref{eq:2D.RACS.condition.DdivU.equiaffine}.
	This confirms that the conditions~\eqref{eq:RACS2ACSIS.std} hold true and therefore the claim holds for abundant homothetically Blaschke normalisations.
\end{proof}

With Lemma~\ref{lem:2D.standard.equiaffine} at hand, we return to the proof of Theorem~\ref{thm:2D}.
It suffices to prove the following lemma:
\begin{lem}\label{lem:2D.conformal.compatibility}
	The ansatz~\eqref{eq:ansatz.2D} is compatible with conformal rescalings.
\end{lem}
\begin{proof}
	For the metrics and \eqref{eq:ansatz.2D.1}, the compatibility under conformal transformations is immediately clear, and for \eqref{eq:ansatz.2D.2} it is confirmed by a direct computation.
	Indeed, as shown in Section~\ref{sec:conformal}, we have the transformation rules $G\to G'=\Omega^2G$,
	\begin{align*}
		\hat C(X,Y) &\to \hat C(X,Y)-\Upsilon(X)Y-\Upsilon(Y)X-\hat\Upsilon\,G(X,Y)\,,
		\\
	u&\to u- \Upsilon\,,\\
	U&\to\Omega^2U
	\end{align*}
	for abundant hypersurface normalisations, where $\Upsilon= d \ln |\Omega |$.
	Moreover,
	\begin{align*}	
		\hat S(X,Y) &\to \hat S(X,Y)\,,
		\\
		t &\to t-3\ln|\Omega|\quad\text{(up to an irrelevant constant)}
	\end{align*}
	for abundant manifolds.
	The validity of~\eqref{eq:ansatz.2D} after the transformation follows, proving the claim.
\end{proof}

Theorem~\ref{thm:2D} now follows from Lemma~\ref{lem:2D.standard.equiaffine} together with Lemma~\ref{lem:2D.conformal.compatibility} and the 
conformal invariance of the abundantness conditions \eqref{eq:ansatz.2D}, see Remark~\ref{rmk:abundant_hyp}.
Indeed, for any given abundant manifold, we may apply a conformal transformation (with rescaling factor $\Omega$) to transform it into a system with $t=0$. We then use Lemma~\ref{lem:2D.standard.equiaffine} to obtain the corresponding abundant Blaschke normalisation. Applying a conformal transformation of this hypersurface normalisation, with rescaling factor $\Omega^{-1}$, we obtain an abundant hypersurface normalisation whose metric, cubic and Weingarten tensor satisfy~\eqref{eq:ansatz.2D}. This abundant hypersurface is as claimed in part~(i) of Theorem~\ref{thm:2D}. One similarly argues for part (ii) of Theorem~\ref{thm:2D}.

\section{Applications and examples}\label{sec:ex}

We now discuss a selection of examples, applying the general theory developed in the previous sections.
While this selection of examples is not exhaustive, we believe that it offers a useful illustration of the methodology presented in this paper.

\subsection{Quadric affine hypersurfaces and the isotropic harmonic oscillator}

We begin our discussion with \emph{quadric-type relative hypersurfaces}, cf.\ Definition~\ref{defn:quadric}.
Recall that the \emph{non-degenerate} isotropic harmonic oscillator system is defined on Euclidean space $(M,g_0)$ and that it is  characterised by the vanishing of the structure tensor, $T\equiv0$ \cite{KSV2023}.
As a direct consequence of a statement in \cite[Chapter 7]{SSV1991}, we obtain the following lemma.
We recall that we have proven that the ansatzes \eqref{eq:ansatz.RACS} and~\eqref{eq:ansatz.2D} indeed establish a 1-to-1 correspondence between abundant hypersurface normalisations and abundant manifolds, compare Theorems \ref{thm:main.1}, \ref{thm:main.2} and \ref{thm:2D}.
\begin{lem}\label{lem:HO.class}
	Abundant hypersurface normalisations are quadrics if and only if the associated abundant manifold satisfies $S=0$ under the correspondence \eqref{eq:ansatz.RACS} $(n\geq3)$ or \eqref{eq:ansatz.2D} $(n=2)$.
\end{lem}
\begin{proof}
	To prove the ``$\Rightarrow$'' part, note that due to Definition~\ref{defn:quadric} and the discussion following it, we have $U=0$, and hence $S=0$.
	
	For the ``$\Leftarrow$'' part, simply note that $S=0$ implies $U=0$, and then the statement follows from the discussion after Definition~\ref{defn:quadric}.
\end{proof}

With this lemma, the following statement is immediately obtained. Recall that the non-degenerate isotropic harmonic oscillator system
is the Hamiltonian system with the Hamiltonian $H=\sum_{i=1}^n (p_i)^2 + V$, where $V=a_0 \sum_{i=1}^n (x^i)^2 + \sum_{i=1}^n a_ix^i + c$, $a_0\neq 0$. 
\begin{cor}\label{cor:HO.class.quadrics}
	An abundant conformally superintegrable system that is conformally equivalent to the non-degenerate isotropic harmonic oscillator system is associated with (an open subset of) a quadric affine hypersurface.
\end{cor}
\begin{proof}
	Indeed, cf.~\cite{KSV2024} or Example~\ref{ex:HO} below, any system in the conformal class of the non-degenerate isotropic harmonic oscillator system satisfies $S=0$. Its associated affine hypersurface therefore satisfies $U=0$ under the correspondence \eqref{eq:ansatz.RACS} $(n\geq3)$ or \eqref{eq:ansatz.2D} $(n=2)$, respectively.
\end{proof}

\begin{ex}\label{ex:HO} 
	The non-degenerate harmonic oscillator system is encoded in the abundant manifold $(\R^n,\sum_{i=1}^n (dx^i)^2,S=0,t=0)$.	
	By Lemma~\ref{lem:HO.class}, its associated relative affine hypersurface lies on a quadric. Moreover, it clearly is a Blaschke hypersurface, since $t=0$.
	It follows, from the uniqueness part of Theorem~\ref{thm:existence.uniqueness}, that the relative affine hypersurface associated to the non-degenerate harmonic oscillator system is an open subset of an elliptic paraboloid.
\end{ex}

From this example, and together with Definition~\ref{defn:RACS.hypersurface} and Lemma~\ref{cor:HO.class.quadrics}, the following statement follows.
\begin{prop}\label{prop:ho.class}
	The elliptic paraboloid is the abundant hypersurface that is associated to the conformal equivalence class of the harmonic oscillator system.
\end{prop}

\subsection{Proper abundant manifolds of constant sectional curvature}
We now consider \emph{proper} (i.e.~$\tau=0$) abundant structures on constant curvature spaces.
We begin with dimensions $n\geq3$.

\begin{lem}\label{lem:A.formula.via.tau}
	The abundant hypersurface normalisation of an abundant manifold of dimension $n\geq3$ satisfies
	\begin{equation}\label{eq:A.formula.via.tau}
		A = \frac13\left(8\,\mathring{\P}^G-\tau\right)
			+ \frac{G}{n(n-1)}\left( 
				\Scal^G-|U|^2+(n-1)(n+2)|u|^2
			\right)\,.
	\end{equation}
\end{lem}
\begin{proof}
	Using~\eqref{eq:A.identity.Ric}, we have
	\[
		\mathring{A} = \frac2{n-2}\left(
			(n-2)\mathring{\P}^g
			-\mathring{\mathscr{C}}
			+(n+2)\left[ C(\hat u)-\frac{n+2}{n}\,|u|^2\,G \right]
		\right)
	\]
	and then, invoking Lemma~\ref{lem:useful.C2Uu.identities},
	\[
		\mathring{A} = 2\left[
				\mathring{\P}^g-\frac1{n-2}\mathring{\mathscr{U}}+U(\hat u)+(u\otimes u-\frac1n\,|u|^2\,G)
			\right]\,.
	\]
	Moreover, \eqref{eq:ACSIS.conds.Ric-Tau} yields, under~\eqref{eq:ansatz.C} and~\eqref{eq:ansatz.G},
	\[
		\tau = 2\mathring{\P}^G+6\left( \frac{1}{n-2}\mathring{\mathscr{U}}-U(\hat u)-u\otimes u+\frac1n\,|u|^2\,G \right)\,.
	\]
	Hence we conclude that
	\[
		\mathring{A} = 2\mathring{\P}^G-\frac26\left( \tau-2\mathring{\P}^G \right)
		= \frac13\left( 8\,\mathring{\P}^G-\tau\right)\,.
	\]
	Now, recalling~\eqref{eq:trA},
	\[
		(n-1)\tr(\hat A) = \Scal^G-|C|^2+(n+2)^2|u|^2 = \Scal^G-|U|^2+(n-1)(n+2)|u|^2
	\]
	since $|C|^2=|U|^2+3(n+2)|u|^2$. The claim follows.
\end{proof}

\begin{prop}\label{prop:proper}
	Abundant properly superintegrable systems on Riemannian manifolds with $n\geq3$ and constant sectional curvature $\kappa$ are associated with improper relative spheres, i.e.~$\hat A=0$.
\end{prop}
\begin{proof}
	Invoking Lemma~\ref{lem:A.formula.via.tau}, we have
	\[
		A = \frac13\left(8\,\mathring{\P}^G-\tau\right)
			+ \frac{G}{n(n-1)}\left( 
				\Scal^G-|U|^2+(n-1)(n+2)|u|^2
			\right)\,.
	\]
	Using~\eqref{eq:ansatz.RACS} and $\Scal^g= n(n-1)\kappa$, we obtain
	\begin{equation}\label{eq:perfect.square}
		9\left( \Scal^G-|U|^2+(n-1)(n+2)|u|^2\right)
		= 9n(n-1)\kappa -|S|^2+(n-1)(n+2)|\grad_g t|^2
		= 0\,,
	\end{equation}
	where in the last step we use Formula (7.10c) from~\cite{KSV2023}.
	Via the contracted Bianchi identity, we also have $\mathring{\P}^g=0$, and therefore we have
	\[
		A = -\frac13\tau = 0\,,
	\]
	since the system is proper, meaning that $\tau$ vanishes.
\end{proof}

\begin{rmk}\label{rmk:proper}
	The statement of Proposition~\ref{prop:proper} has been made for abundant properly superintegrable systems, but extends to \emph{abundant manifolds of dimension $n\geq3$ with $\tau=0$ and such that $\nabla^g$ has constant sectional curvature}, i.e.\ these manifolds are associated with improper relative spheres in the sense of the correspondence of Theorems~\ref{thm:main.1} and~\ref{thm:main.2}. Indeed, a review of the proof of Proposition~\ref{prop:proper} shows that this generalisation holds if and only if Formula (7.10c) of \cite{KSV2023} holds, i.e.\ if
	\[
		9\left( \Scal^G-|U|^2+(n-1)(n+2)|u|^2\right)
		= 0\,.
	\]
	A thorough examination of the relevant parts of~\cite{KSV2023} indeed shows that this identity follows  by differentiation of~\eqref{eq:ACSIS.conds.DDt} and~\eqref{eq:ACSIS.conds.DS} and expressing the skew part of the result by the curvature tensor. One then uses that~\eqref{eq:ACSIS.conds.Weyl} and Equation~\eqref{eq:ACSIS.conds.Ric-Tau}, i.e.\ 
	\[
		\mathring{\mathscr{S}}
		-(n-2) \left( \hat{S}(dt) + dt\otimes dt-\frac1n\,|\grad_g(t)|^2_g\,g \right) = 0. 
	\]
(Note that the latter holds true in the case under consideration as $\tau=0$ and due to the assumption of constant sectional curvature.)
Hence, it follows that the abundant hypersurfaces associated to abundant structures on manifolds of constant sectional curvature are associated to improper relative spheres and satisfy
	\begin{align*}
		\bar\nabla_XY &= \nabla_XY+g(X,Y)\xi\,, \\
		\bar\nabla_X\xi &= 0\,.
	\end{align*}
\end{rmk}

We now prove a converse statement.
\begin{prop}
	Let $(f,\xi)$ be an improper relative sphere that is an abundant hypersurface normalisation with $n\geq3$ and whose Levi-Civita connection $\nabla^G$ has constant curvature $\kappa$.
	Then $(f,\xi)$ corresponds to an abundant manifold with $\tau=0$.
\end{prop}
Note that $G$ is, as always, tacitly assumed to be non-degenerate.
\begin{proof}
	Invoking Lemma~\ref{lem:A.formula.via.tau}, we infer
	\begin{align*}
		\mathring{A} &= \frac13\left(8\,\mathring{\P}^G-\tau\right)  = 0\,,
	\end{align*}
	since the hypersurface by hypothesis satisfies $A=0$.
	Since $G$ has constant sectional curvature $\kappa$, we have, in particular, that $\mathring{\P}^G=0$.
	It hence follows that $\tau=0$.
\end{proof}

We conclude our study of proper abundant manifolds with a brief discussion of the known two-dimensional examples of second-order properly superintegrable systems on flat space \cite{Evans1990}.
We have, indeed, already seen one of the most famous examples, namely the non-degenerate harmonic oscillator, cf.\ Example~\ref{ex:HO}.
We now discuss the Smorodinski-Winternitz systems in dimension two. We leave it to the interested reader to extend the examples (in the straightforward manner) to higher dimensions.
\begin{ex}\label{ex:SW1}
	The so-called Smorodinski-Winternitz I system, labeled by [E1] in \cite{KKPM2001}, is defined on the Euclidean space $(\R_+^2,g=dx^2+dy^2)$ and has the (non-degenerate) potential
	\[ V = a_0\,(x^2+y^2)+\frac{a_1}{x^2}+\frac{a_2}{y^2}+a_3\,. \]
	Computing the structure tensor, see~\cite{KSV2023,KSV2024_bauhaus}, and using~\eqref{eq:ansatz.C}, we obtain the cubic $C$ as
	\[ C = -\frac1x\,dx\otimes dx\otimes dx-\frac1y\,dy\otimes dy\otimes dy\,. \]
	Integrating~\eqref{eq:structure.relative} explicitly, we find
	\[
		\xi = c_1\,\partial_x+c_2\,\partial_y+c_3\,\partial_z
	\]
	for constants $c_1,c_2,c_3\in\R$. By an affine transformation, we transform this into
	\[
		\xi = \partial_z
	\]
	for simplicity. The immersion $f$ then is obtained in the form
	\[
		f(x,y) = \left(\begin{array}{c}
			C_{10}\ln(x) + C_{11}\ln(y) + C_{12} \\
			C_7\ln(x) + C_8\ln(y) + C_9 \\
			\frac14\left(x^2+y^2\right) + C_4\ln(x) + C_5\ln(y) + C_6
		\end{array}\right)\,,
	\]
	where the $C_k$ are integration constants. An appropriate choice of these constants yields, for instance,
	\[
		f(x,y) = \left(\begin{array}{c}
			\ln(x) \\
			\ln(y) \\
			\frac14(x^2+y^2)
		\end{array}\right)\,.
	\]
\end{ex}

\begin{ex}\label{ex:SW2}
	We now consider the two-dimensional Smorodinski-Winternitz II system, labeled by [E2] in \cite{KKPM2001}, which is defined on the Euclidean space $(\R^2,g=dx^2+dy^2)$ and has the (non-degenerate) potential
	\[ V = a_0\,(4x^2+y^2)+a_1\,x+\frac{a_2}{y^2}+a_3\,. \]
	We obtain the cubic as
	\[ C = -\frac1y\,dy\otimes dy\otimes dy\,, \]
	and then find the immersion, up to a suitable choice of the integration constants,
	\[
		f(x,y) = \left(\begin{array}{c}
			x \\
			\ln(y) \\
			\frac{x^2}{2} + \frac{y^2}{4}
		\end{array}\right)
	\]
	with $\xi=\partial_z$.
\end{ex}

We also compute a non-flat example, the so-called \emph{generic system} on the $2$-sphere.

\begin{ex}\label{ex:generic}
	We now consider the two-dimensional sphere $\mathds S^2\subset\R^3$, on which the so-called \emph{generic system}, labeled by [S9] in \cite{KKPM2001}, is defined. Using standard coordinates $(X,Y,Z)$ on $\R^3$ and using stereographic projection,
	\begin{equation}\label{eq:stereographic}
		(X,Y,Z)=\left(\frac{2x}{1+x^2+y^2},\frac{2y}{1+x^2+y^2},\frac{x^2+y^2-1}{x^2+y^2+1}\right),
	\end{equation}
	this system is defined on $\left(U\subset\mathds S^2,g=\frac{4(dx^2+dy^2)}{(1+x^2+y^2)^2}\right)$ by the (non-degenerate) superintegrable potential (where defined)
	\[
		V = a_0\,\frac{(x^2 + y^2 + 1)^2}{(x^2 + y^2 - 1)^2}
		+ a_1\,\frac{(x^2 + y^2 + 1)^2}{x^2} + a_2\,\frac{(x^2 + y^2 + 1)^2}{y^2} + a_3\,.
	\]
	In coordinates $(X,Y,Z)$ on $\R^3$, this reads
	$	V = \frac{a_0}{Z^2}	+ \frac{4a_1}{X^2} + \frac{4a_2}{Y^2} + a_3	$\,.
	Analogously to the previous examples, we then find the immersion up to a suitable choice of integration constants. We obtain, with $\xi=\partial_z$,
	\[
		f(x,y) = \left(\begin{array}{c}
			\ln \frac{|x|}{|x^2+y^2-1|} \\
			\ln \frac{|y|}{|x^2+y^2-1|} \\
			\frac12\ln\frac{x^2+y^2+1}{|x^2+y^2-1|}
		\end{array}\right)\,.
	\]
\end{ex}

\subsection{Graph normalisations}

We conclude the paper with a  study of the hypersurface normalisations that correspond locally to the vertical normalisation $\xi =\partial_{n+1}$ of a graph of a function $F:\mathbb{R}^n \to \mathbb{R}$.
To this end, we utilize a result in \cite{NP1987}, see also~\cite[\S~7.3]{SSV1991}.
A relative affine hypersurface normalisation is called a \emph{graph normalisation} if there is a (smooth) function $F:M\to\R$ such that \eqref{eq:structure.relative} becomes
\begin{align*}
	\bar\nabla_XY &= \nabla_XY+(\nabla^2F)(X,Y)\,\xi\,,
	\\
	\bar\nabla_X\xi &= 0\,.
\end{align*}

The existence of a graph normalisation for a relative affine hypersurface has been characterised in~\cite[Example~3 and Proposition~4]{NP1987}:

\begin{prop}\label{prop:graph.immersion} %
	Let  $(f:M\to\R^{n+1},\Xi)$ be a relative affine hypersurface co-norma\-li\-sation.
	Then the Weingarten tensor vanishes, $A=0$, if and only if the affine hypersurface normalisation is affinely equivalent to a graph normalisation.
\end{prop}

We then obtain immediately the following corollary. 

\begin{cor}
	Let $(M,g)$, $n\geq3$, be a simply connected oriented Riemannian mani\-fold of constant sectional curvature.
	Assume that $(M,g)$ is equipped with an abundant (second-order) properly (maximally) superintegrable system.
	Then, up to affine transformations, the relative affine hypersurface normalisation specified uniquely in Theorem~\ref{thm:main.1}, is a graph normalisation.
\end{cor}
\begin{proof}
	The statement follows directly from Proposition~\ref{prop:proper} together with Proposition~\ref{prop:graph.immersion}. Indeed, up to affine transformations, it follows that the relative affine hypersurface normalisation specified uniquely in Theorem~\ref{thm:main.1}, is locally a graph normalisation.
\end{proof}

The following example shows that, in this corollary, the dimensional requirement $n\geq3$ is necessary.
\begin{ex}
	Consider the abundant superintegrable system on the $2$-sphere which is labeled by [S7] in~\cite{KKPM2001}, with the notation as in Example~\ref{ex:generic}. We again use the stereographic projection~\eqref{eq:stereographic}.
	This system is defined on
	\[
		\left(U\subset\mathds S^2,g=\frac{4(dx^2+dy^2)}{(1+x^2+y^2)^2}\right)
	\]
	by the (non-degenerate) superintegrable potential
	\[
		V = a_0\,\frac{X}{\sqrt{Y^2+Z^2}}
			+ a_1\,\frac{Y}{Z^2\sqrt{Y^2+Z^2}}
			+ a_2\,\frac1{Z^2}
			+ a_3\,
	\]
	(where defined).
	A direct computation shows that the Weingarten tensor is non-vanishing and in particular we find
	\[
		\tr\hat A = \frac{2}{1-X^2}\ne0\,.
	\]
	The requirements of Proposition~\ref{prop:graph.immersion} are therefore not met for the abundant structure arising from this system, cf.\ Section~\ref{sec:surfaces}.
\end{ex}

We also remark the following:
\begin{cor}
	For a relative affine hypersurface co-normalisation $(f:M\to\R^{n+1},\Xi)$ the following are equivalent:
	\begin{enumerate}[label=(\roman*)]
		\item\label{it:graph.immersion}
			$(f,\Xi)$ is a graph (co-)normalisation.
		\item\label{it:A=0}
			Its Weingarten tensor vanishes.
		\item\label{it:nabla.flat}
			Its induced connection $\nabla$ is flat.
		\item\label{it:nabla*.flat}
			Its induced dual connection $\nabla^*$ is flat.
		\item\label{it:Hessian}
			Its associated hypersurface metric $g$ is Hessian (with respect to $\nabla$).
	\end{enumerate}
\end{cor}
Note that a metric is Hessian if there exists a flat, torsion-free connection such that the metric is locally the Hessian of a function with respect to this connection. 
Such a function is called a Hessian potential. 
\begin{proof}
	For \ref{it:graph.immersion}$\Leftrightarrow$\ref{it:A=0}, recall Proposition~\ref{prop:graph.immersion}.
	For \ref{it:A=0}$\Leftrightarrow$\ref{it:nabla.flat}, recall Equation~\eqref{eq:Ric.via.A} and observe that the flatness of the induced connection $\nabla$ is equivalent to $A=0$. 
	Similarly, inspecting~\eqref{eq:Riem.A}, we find that $\nabla^*$ is flat if and only if $A=0$, proving \ref{it:A=0}$\Leftrightarrow$\ref{it:nabla*.flat}.
	
	For the final equivalence, note that for a graph normalisation, the above implies that both $\nabla$ and $\nabla^*$ are flat, defining a so-called $g$-dually flat structure. It follows that~$g$ is Hessian with respect to $\nabla$, $g=\nabla^2\phi$ for some $\phi\in\mathcal C^\infty(M)$, e.g.~\cite{Amari&Armstrong,Shima}. 
	In fact, the integrability condition for the existence of a (local) Hessian potential is provided by the Codazzi equation \eqref{eq:integrability.gauss.2} for $A=0$. 
	(We remark that $g$ is also Hessian with respect to~$\nabla^*$.)
\end{proof}

In the remainder of this section, we use Proposition~\ref{prop:graph.immersion} to characterise the existence of a graph normalisation for abundant manifolds.
We begin with the case $n\geq3$: with the help of Lemma~\ref{lem:A.formula.via.tau}, and~\eqref{eq:ansatz.C}, we obtain that $A=0$ is equivalent to
\begin{subequations}\label{eq:curvature.hessian}
	\begin{align}
		\mathring{\P}^G &= \frac18\,\tau\,,
		\label{eq:curvature.hessian.1}
		\\
		\Scal^G &= \frac19\left( |S|^2-(n-1)(n+2)|\grad_gt|^2 \right).
		\label{eq:curvature.hessian.2}
	\end{align}
\end{subequations}
In the case of dimension $n=2$, due to~\eqref{eq:ansatz.2D.2} and \eqref{eq:sha}, the condition $A=0$ is equivalent to
\begin{align}
	(\nabla^g)^2t -\frac12\Delta^gt\,g
	&= \frac12\left( \sha-\frac23\,S(\grad_gt) \right),
	\label{eq:hessian.2D.DDt}
	\\
	\Scal^g
	&=\frac19\left( |S|^2 -4\,|\grad_gt|^2 \right).
	\label{eq:hessian.2D.curvature}
\end{align}
Note that we may rewrite~\eqref{eq:hessian.2D.DDt}, using~\eqref{eq:DDt.2D}, as
\begin{equation}
	\bigtau =
	\frac23\,\sha
	-\frac89\left(
		S(\grad_gt)+dt\otimes dt-\frac12\,|\grad_gt|^2\,g
	\right)\,.
	\label{eq:hessian.2D.tau}
\end{equation}
We therefore conclude:
\begin{thm}\label{thm:graph.immersions}
	Let $(M,g,S,t)$ be an abundant manifold.
	Then its associated relative affine hypersurface normalisation is an affine graph normalisation if and only if
	\begin{itemize}[left=5pt]
		\item for $n\geq3$: the conditions \eqref{eq:curvature.hessian} hold for the trace-free Schouten tensor and scalar curvature of $g$, respectively.
		\item for $n=2$: the scalar curvature of $g$ satisfies \eqref{eq:hessian.2D.curvature}, and the tensor~$\tau$ satisfies~\eqref{eq:hessian.2D.tau}. 
	\end{itemize}	
	
\end{thm}

We observe that, for dimensions $n\geq3$, \eqref{eq:curvature.hessian} are conditions between $(S,t)$, which is part of the abundant structure, and the underlying Riemannian metric. For dimension $n=2$, on the other hand, we have an analogous curvature condition, namely~\eqref{eq:hessian.2D.curvature}, but also the condition~\eqref{eq:hessian.2D.tau}.
Regarding the latter condition, we recall that, for an abundant conformally superintegrable system, the tensor $\tau$ encodes whether the system gives rise to a second-order (maximally) superintegrable system in the classical sense (existence of constants of the motion for all Hamiltonian trajectories instead of only trajectories on the zero locus of the Hamiltonian).

\appendix

\section{Relationship of the equations~\eqref{eq:RACS.conditions.DU} and~\eqref{eq:ACSIS.conds.DS}}\label{app:DU.DS}
In this section we detail some computations mentioned in Sections~\ref{sec:proof.acsis2racs} and~\ref{sec:proof.racs2acsis}.
We begin with a computation of the condition~\eqref{eq:ACSIS.conds.DS}.
Recall that the derivative of $S$ is defined in~\eqref{eq:ACSIS.conds.DS} via the auxiliary tensor field $S_1$, defined in~\eqref{eq:S1}.
We hence have to compute
\[
	\nabla^gS = \frac13\,\Pi_{\mathrm{Sym}_0^3} S_1\,,
\]
cf.~\eqref{eq:ACSIS.conds.DS}. To this end, we first symmetrise in the first three arguments of $S_1$,
\[
	\Phi(X,Y,Z,W) := \Pi_{\mathrm{Sym}^3} S_1(X,Y,Z,W)\,,
\]
obtaining
\begin{multline*}
	\Phi(X,Y,Z,W)
	= \frac13 \Big( \mathfrak{S}(X,W,Y,Z) + \mathfrak{S}(Y,W,X,Z) + \mathfrak{S}(Z,W,X,Y) \Big) 
	\\
	+S(X,Y,W)Z(t) + S(Y,Z,W)X(t) + S(X,Z,W)Y(t) + S(X,Y,Z)W(t)
	\\
	+\frac{4}{3(n-2)}\Big(
		\mathscr S(Y,Z)g(X,W)+\mathscr(X,Z)g(Y,W)+\mathscr{S}(X,Y)g(Z,W)
	\Big)
	\\
	-\Big(
		S(Y,Z,\grad_g(t))g(X,W) + S(X,Z,\grad_g(t))g(Y,W)
	\\ + S(X,Y,\grad_g(t))g(Z,W)
	\Big)\,.
\end{multline*}
Next, we take the trace of $\Phi$ (in any pair of arguments of $\Phi$ over which we have symmetrised).
For instance, taking the trace in the second and third argument of~$\Phi$, we obtain 
\begin{equation*}
	\phi(X,W) := \tr_g\Big(\Phi(X,-,Y,-)\Big)
	=\frac23\frac{n+2}{n-2}\,\mathscr{S}(X,W) + \frac{4}{3(n-2)}\,|S|_g^2\,g(X,W)\,.
\end{equation*}
Hence, we conclude that
\begin{multline*}
	\Pi_{\mathrm{Sym}^3} (g\otimes\phi)(X,Y,Z,W)
	\\
	= \frac29\frac{n+2}{n-2}
	\left(
		g(X,Y)\mathscr{S}(Z,W)+g(Y,Z)\mathscr{S}(X,W)+g(Z,X)\mathscr{S}(Y,W)
	\right)
	\\
	+\frac{4\,|S|_g^2}{9(n-2)}\left(
		g(X,Y)g(Z,W) + g(Y,Z)g(X,W) + g(Z,X)g(Y,W)
	\right)\,.
\end{multline*}
We therefore arrive at
\begin{align}
	3\,\nabla^g_WS(X,Y,Z)
	&= \frac13 \left( \mathfrak{S}(X,W,Y,Z) + \mathfrak{S}(Y,W,X,Z) + \mathfrak{S}(Z,W,X,Y) \right)
	\nonumber
	\\ &\quad +S(X,Y,W)Z(t) + S(Y,Z,W)X(t)
	\nonumber
	\\ &\qquad\quad + S(X,Z,W)Y(t) + S(X,Y,Z)W(t)
	\nonumber
	\\ +\frac{4}{3(n-2)}\Big(
		&\mathscr S(Y,Z)g(X,W)+\mathscr S(X,Z)g(Y,W)+\mathscr{S}(X,Y)g(Z,W)
		\Big)
		\nonumber
	\\ &-\Big(
		S(Y,Z,\grad_g(t))g(X,W) + S(X,Z,\grad_g(t))g(Y,W)
		\nonumber
		\\ &\qquad\qquad\qquad\qquad\qquad\qquad + S(X,Y,\grad_g(t))g(Z,W)
		\Big)
		\nonumber
	\\ -\frac{3}{n+2}\bigg[
	&\frac29\frac{n+2}{n-2}\left(
				g(X,Y)\mathscr{S}(Z,W)+g(Y,Z)\mathscr{S}(X,W)+g(Z,X)\mathscr{S}(Y,W)
			\right)
	\nonumber
	\\
	&+\frac{4\,|S|_g^2}{9(n-2)}\left(
	g(X,Y)g(Z,W) + g(Y,Z)g(X,W) + g(Z,X)g(Y,W)
	\right)
	\bigg]\,.
\label{eq:DS.full}
\end{align}

\subsection{Equation~\eqref{eq:RACS.conditions.DU} follows from~\eqref{eq:ACSIS.conds.DS} under the hypotheses}\label{app:DU}
In Section~\ref{sec:proof.acsis2racs}, we need to show that \eqref{eq:RACS.conditions.DU} holds under the ansatz~\eqref{eq:ansatz.C} and~\eqref{eq:ansatz.G}, given the hypothesis.
This reduces to showing that
\[
	\Pi_{\mathrm{Sym}^4_0} \nabla^G U
	= \Pi_{\mathrm{Sym}^4_0}\left[
		\mathfrak U + 4\,U\otimes u
	\right]
\]
is implied by~\eqref{eq:ACSIS.conds.DS} under the ansatz.
This is readily verified as follows. First, symmetrise~\eqref{eq:ACSIS.conds.DS} in all four arguments,
\begin{align*}
	3\Pi_{\mathrm{Sym}^4} \nabla^g S &:=
	\Pi_{\mathrm{Sym}^4}\bigg(
	\mathfrak{S} + 4\,S\otimes dt
	+\frac{4}{n-2}\,\mathscr{S}\otimes g - 3\,S(\grad_g(t))\otimes g
	\\ &\qquad\qquad\qquad
	-\frac{2}{n-2}\,\mathscr{S}\otimes g
	-\frac{4\,|S|_g^2}{(n-2)(n+2)}\,g\otimes g
	\bigg)\,.
\end{align*}
Then project onto the trace-free part, obtaining
\begin{equation}
	\Pi_{\mathrm{Sym}_0^4} \nabla^g S = \frac13\,\Pi_{\mathrm{Sym}_0^4} \left( \mathfrak{S} + 4\,S\otimes dt \right)\,.
\end{equation}
Now, using~\eqref{eq:ansatz.C} to replace $S$ by $U$, $dt$ by $u$, etc., we arrive at~\eqref{eq:RACS.conditions.DU}.

\subsection{Equation~\eqref{eq:ACSIS.conds.DS} follows from~\eqref{eq:RACS.conditions.DU} under the hypotheses}\label{app:DS}

In Section~\ref{sec:proof.racs2acsis}, we need to prove that~\eqref{eq:ACSIS.conds.DS} holds.
Note that under the hypotheses, Equations~\eqref{eq:A.identity.Codazzi} and~\eqref{eq:RACS.conditions.DU} hold.
Under the ansatz~\eqref{eq:ansatz.RACS}, it is then easily proven that
\[
	3\,\Pi_{\mathrm{Sym}^4_0}\nabla^g S=\Pi_{\mathrm{Sym}_0^4} \left( \mathfrak{S} + 4\,U\otimes u \right)\,,
\]
consistent with~\eqref{eq:DS.full}. The proof is similar to the computation used in the previous section.
It hence remains to consider two identities, namely
\begin{equation}\label{eq:condition.Codazzi.DS}
	\Pi_{\mathrm{Codazzi}_0}\nabla^g S=0
\end{equation}
and
\begin{equation}\label{eq:condition.div.S}
	3\div_g(S) = \frac{2n}{n-2}\,\mathscr{S}-nS(\grad_g(t)) -\frac{2}{n-2}\,|S|_g^2\,g\,,
\end{equation}
since due to the symmetries, the condition~\eqref{eq:ACSIS.conds.DS} has only one independent trace.

For each of these two conditions, we need to confirm that they are consistent with~\eqref{eq:ACSIS.conds.DS}, and that they follow under the hypotheses in Section~\ref{sec:proof.racs2acsis}.

We begin with the first condition, i.e.~\eqref{eq:condition.Codazzi.DS}.
A direct computation using~\eqref{eq:DS.full} implies
\begin{multline*}
	3\,\left(\nabla_WS(X,Y,Z)-\nabla_XS(W,Y,Z)\right)
	\\
	= Q(X,Z)g(Y,W) - Q(W,Z)g(X,Y)
	\\
	+ Q(X,Y)g(Z,W) - Q(Y,W)g(X,Z)\,,
\end{multline*}
where
\[
	Q = \frac{2}{n-2}\,\mathscr{S} - S(\grad_g(t))\,,
\]
and therefore~\eqref{eq:condition.Codazzi.DS} is consistent with~\eqref{eq:ACSIS.conds.DS}. It is satisfied as well, as it immediately follows from~\eqref{eq:A.identity.Codazzi}, which holds due to the hypotheses in Section~\ref{sec:proof.racs2acsis}.
For the condition~\eqref{eq:condition.div.S}, we compute, using again~\eqref{eq:DS.full},
\begin{align*}
	3\div_g(S) &= 3n\,\left(Q-\frac1n\,g\,\tr_g(Q) \right)
	\\ &=
	\frac{2n}{n-2}\,\mathscr{S}-nS(\grad_g(t)) -\frac{2}{n-2}\,|S|_g^2\,g\,.
\end{align*}
Combining~\eqref{eq:A0}, \eqref{eq:trA} and~\eqref{eq:A.identity.Ric} with the equations in Lemma~\ref{lem:useful.C2Uu.identities}, with the identity
\[
	C(\hat u) = U(\hat u)+|u|_G^2\,G+2\,u\otimes u\,,
\]
and with~\eqref{eq:RACS.conditions.Du}, we obtain
\begin{align*}
	3\,\div_g(S) &= \div_G(U) = \div_G(C)-2\nabla u-\div_G(u)\,G
	\\ &=
	-n\,\left[
		\mathring{\P}-\frac1{n-2}\mathring{\mathscr{U}}+U(\hat u)+\left(u\otimes u-\frac1n\,G\,|u|_G^2\right)
		-\nabla u+\frac1n\,G\,\div_G(u)
	\right]
	\\ &=
	-n\,\left[
		-\frac{2}{n-2}\mathring{\mathscr{U}}+U(\hat u)
	\right]
	\,,
\end{align*}
and hence~\eqref{eq:condition.div.S} holds and is consistent with~\eqref{eq:DS.full}.
We conclude that~\eqref{eq:ACSIS.conds.DS} is satisfied, and thus the claim is proven.

\bigskip

\section*{Acknowledgements}

AV is grateful towards Wolfgang Schief and Jonathan Kress for discussions. AV was funded by Deutsche Forschungsgemeinschaft (DFG, German Research Foundation) through the project grant 540196982. AV also acknowledges support by the \emph{Forschungsfonds} of the Department of Mathematics at the University of Hamburg.
Research of VC is funded by the DFG 
under Germany's Excellence Strategy, EXC 2121 ``Quantum Universe'' 390833306 and under -- SFB-Gesch\"aftszeichen 1624 -- Projektnummer 506632645.

\section*{Declarations}
	
Data availability statement: the paper has no associated data.
	
The authors have no relevant financial or non-financial interests to disclose.

\bibliographystyle{alphaabbrv}
\bibliography{affine-hypersurfaces}

\end{document}